\documentclass{article}
\usepackage{arxiv}

\usepackage[utf8]{inputenc} 
\usepackage[T1]{fontenc}    
\usepackage{booktabs}       

\usepackage{nicefrac}       
\usepackage{microtype}      
\usepackage{graphicx}

\usepackage[labelfont=bf]{caption}

\usepackage{algorithm}
\usepackage{algpseudocode}

\usepackage{amsmath,amsfonts,amssymb,amsthm,booktabs,color,epsfig}
\theoremstyle{plain}

\usepackage{enumitem}
\usepackage{tikz}
\usepackage{pgfplots}
\pgfplotsset{compat=newest}
\usepackage{subfig}
\usepackage{xcolor}

\definecolor{myblue}{HTML}{1f77b4}
\definecolor{myorange}{HTML}{ff7f0e}

\newcommand*{\VEC}[1]  {\ensuremath{\boldsymbol{#1}}}
\newcommand*{\MAT}[1]  {\ensuremath{\boldsymbol{#1}}}

\DeclareMathOperator{\cov}{cov}
\DeclareMathOperator{\tr}{tr}
\DeclareMathOperator{\vvec}{vec}
\DeclareMathOperator{\var}{var}

\theoremstyle{plain}
\newtheorem{theorem}{Theorem}
\newtheorem{proposition}{Proposition}
\newtheorem{lemma}{Lemma}
\newtheorem{corollary}{Corollary}
\theoremstyle{definition}
\newtheorem{definition}{Definition}
\newtheorem{remark}{Remark}

\usepackage{hyperref}       
\usepackage{url}            
\usepackage{float}
\usepackage{comment}
\usepackage{multirow}
\begin{document}

\title{On Elliptical and Inverse Elliptical Wishart distributions}

\author{
  Imen Ayadi \thanks{\texttt{imenayadi3005@gmail.com}} , Florent Bouchard\thanks{\texttt{florent.bouchard@centralesupelec.fr}} , Frederic Pascal\thanks{\texttt{frederic.pascal@centralesupelec.fr}} \\
  Université Paris-Saclay, CNRS, CentraleSupélec,\\ Laboratoire des Signaux et Systèmes \\ 91190, Gif-sur-Yvette, France}



\maketitle

\begin{abstract}
This paper deals with the Elliptical Wishart and Inverse Elliptical Wishart distributions, which play a major role when handling covariance matrices.
Similarly to multivariate elliptical distributions, these form a large family of covariance distributions, encompassing, \textit{e.g.}, the Wishart or $t$-Wishart ones.
Our first major contribution is to derive a stochastic representation for Elliptical Wishart and Inverse Elliptical Wishart matrices.
This later enables us to obtain various key statistical properties of these distributions such as expectations, variances, and Kronecker moments up to any orders.
The stochastic representation provides an efficient method for generating random matrices from these distributions.
Additionally, Elliptical Wishart distributions have a specific dependence structure — a random scale shared by all
columns — which strongly influences their asymptotic behavior.
Lastly, the practical relevance of Elliptical Wishart distributions, particularly the $t$-Wishart distribution, is demonstrated through a fitting experiment conducted on real electroencephalographic data. This experiment highlights their effectiveness in accurately modeling real covariance matrices.
\end{abstract}

\keywords{
Covariance matrices \and
Elliptical Wishart \and
Inverse Elliptical Wishart \and
Normalized Wishart \and
Stochastic representation \and
Kronecker moments\and
Electroencephalography fitting.
}

\section{Introduction\label{sec:intro}}

Covariance matrices are a crucial tool for statistical signal processing and machine learning.
They are very useful in several fields such as biosignals~\cite{kalunga2016online}, radar~\cite{pascal2008performance}, image processing~\cite{portilla2003image}, MIMO~\cite{lopez2015eigenvalue} or change detection~\cite{mian2024online}.
%
Statistics over covariance matrices have also been exploited in various contexts.
Some of these matrix-variate distributions, such as Wishart and Inverse Wishart distributions, have, for instance, been leveraged as priors in Bayesian learning; see, \textit{e.g.},~\cite{bidon2007bayesian}.
When performing classification or clustering with covariance matrices as features, the notorious Fréchet mean has been extensively exploited; see, \textit{e.g.},~\cite{kalunga2016online}.
It yielded the so-called Riemannian Gaussian distribution~\cite{said2017riemannian}.
Recently, for classification, \cite{ayadi2023t} rather proposed to exploit an extension to the Wishart distribution: the $t$-Wishart distribution, part of the Elliptical Wishart family of distributions~\cite{caro2014generalized}.

This paper deals with such distributions over the space of covariance matrices, which is the manifold of symmetric positive definite (SPD) matrices.
%
The most basic and best-known distribution is the Wishart distribution~\cite{wishart1928generalised}, which basically corresponds to the distribution of the (scaled) sample covariance matrices (SCM) of some random Gaussian vectors.
The Inverse Wishart distribution corresponds to the one of SPD matrices whose inverses follow the Wishart distribution.
Hence, by definition, these two distributions play a central role as soon as random matrices or vectors are involved.
They have thus been extensively studied, and most of their properties are very well known -- see, \textit{e.g.},~\cite{bilodeau1999theory,gupta2018matrix}.
In particular, the most important properties -- when it comes to characterizing and exploiting a distribution -- have been derived, such as the expectation, variance, or Kronecker moments.

Due to their significance, several extensions of these distributions have been proposed.
Notably, multiple Wishart-related distributions were introduced afterward. Among these, the Elliptical Wishart distributions represent a compelling and overlooked family of distributions over covariance matrices~\cite{teng1989generalized,caro2014generalized}. By generalizing the Wishart distribution, they offer a powerful extension, much like how multivariate elliptical distributions~\cite{ollila2012complex} generalize the multivariate Gaussian distribution. This generalization provides them with great potential in modeling intrinsic non-Gaussianity in signal processing problems, wherein extremely aberrant recordings and even mislabeling are frequent. Thus, embracing this family can unlock new insights in statistical modeling.
Although they were introduced a long time ago, the number of studies considering the Elliptical Wishart and their inverse, the so-called Inverse Elliptical Wishart distributions, is limited.
These are thus not yet very well known.
Some of their statistical properties have been obtained: their probability density functions (PDF)~\cite{anderson1958introduction}; the characteristic function of Elliptical Wishart distributions~\cite{bekker2015wishart}; the joint probability density of eigenvalues, moments of the trace and determinant~\cite{bekker2015wishart}; and the expectation of zonal polynomials and cumulative density function (CDF)~\cite{caro2016matrix}.
The information geometry and maximum likelihood estimator for the Elliptical Wishart distributions have also been derived in~\cite{ayadi2023elliptical} and \cite{ayadi2025elliptical}, respectively. As an application, \cite{ferreira2020advances} investigated the channel capacity within the MIMO environment in uncorrelated scenarios and showed that the capacity of the system is higher in the case of an underlying $t$-Wishart distribution, which belongs to the family of Elliptical distributions. Later, \cite{ayadi2023t} designed a Bayesian classifier that leverages Elliptical Wishart distributions in the framework of Discriminant Analysis. The obtained classification method was tested on real electroencephalographic data and outperformed current methods. Furthermore, \cite{ayadi2025elliptical} proposed a clustering method based on Elliptical Wishart distributions, which showed competitive results to the state-of-the-art algorithm on hyperspectral imaging data.\\
This paper focuses on deepening the study of the statistics of the Elliptical Wishart and Inverse Elliptical Wishart distributions through the following major contributions:
\vspace{-0.02cm}
\begin{itemize}
    \item Deriving stochastic representations for random matrices drawn from Elliptical and Inverse Elliptical Wishart distributions.
    This is clearly the flagship result of this paper, as other contributions are derived from it.
    \item Deriving some statistical properties of the Normalized Wishart distribution: its expectation, its variance, the expectation and variance of its inverse, and its characteristic function.
    \item Obtaining the main missing statistical properties for these distributions: expectation, variance, and Kronecker moments up to any order.
    Notice that before this paper, Kronecker moments of the Wishart distribution have only been derived up to order 4~\cite{sultan1996moments}.
    They are now also known up to any order.
    \item Investigating the dependence structure of Elliptical Wishart distributions, which makes them robust at the scale of the covariance.
    \item Studying the asymptotic behavior of Elliptical Wishart distributions.
    \item Developing an algorithm to efficiently -- from a computational point of view -- generate Elliptical Wishart and Inverse Elliptical Wishart random matrices.
    \item Performing a fitting experiment with various moments on real electroencephalographic (EEG) data, showing the practical interest of these elliptical extensions -- especially the $t$- one -- as compared to the usual Wishart distribution.
\end{itemize}

\begin{remark}
    Let us emphasize that, while rather oriented on statistical aspects of Elliptical Wishart distributions, these contributions are indeed particularly relevant to the information theory, signal processing and machine learning communities.
    Indeed, the stochastic representation (Theorem~\ref{thm:stoc rep}) along with the expectation and variance from Proposition~\ref{prop:uniform} are exploited to obtain the Fisher information metric in~\cite{ayadi2023elliptical,ayadi2025elliptical}.
    These are also leveraged to investigate the performance of the maximum likelihood estimator in~\cite{ayadi2025elliptical}.
\end{remark}

The article is organized as follows.
Section~\ref{sec:preliminaries} provides a brief reminder about basic statistical tools and about multivariate elliptical distributions.
Section~\ref{sec:review} gives some background about the Wishart and Inverse Wishart distributions, and their elliptical extensions.
Then, Section~\ref{sec:statistics} derives the paper's main theoretical contributions: the stochastic representation along with the resulting statistical properties of the Elliptical Wishart and Inverse Elliptical Wishart distributions. Section ~\ref{sec:compare_scm} unveils the dependence structure of the Elliptical Wishart distributions by comparing them to the distribution of the sample covariance matrix of independent and identically distributed (IID)  elliptical vectors. Section~\ref{sec:asymptotic_elliptical_wishart} establishes the asymptotic behavior of Elliptical-Wishart-distributed matrices.  Section~\ref{sec:application} contains the algorithm to draw random matrices, along with numerical fitting experiments on real EEG data.
The code to reproduce these experiments is provided in \href{https://github.com/ImenAyadi/fitting-_EEG.git}{https://github.com/ImenAyadi/fitting-\_EEG.git}.
Finally, concluding remarks and perspectives are drawn in Section~\ref{sec:conclusion}.

\section{Preliminaries}
\label{sec:preliminaries}
This section is dedicated to outlining the key notations used in the paper and presenting some previous findings on the different distributions under consideration before delving into the contributions of this work.

\subsection{Notations}
\label{subsec:notations}
Vectors (respectively matrices) are denoted by bold-faced lowercase letters (respectively uppercase letters).
$\otimes$, $\tr(\cdot)$, $|\cdot|$, $\|\cdot\|_{2}$, $\|\cdot\|$, and $\rho(\cdot)$ correspond to the Kronecker product, trace, determinant, $\ell_2$ norm, spectral norm and spectral radius, respectively.
$\MAT A^\top$ denotes the transpose of the matrix $\MAT A$.
$\vvec(\cdot)$ is the operator that transforms a matrix $p\times n$ into a vector of length $pn$, using column-wise concatenation.
Moreover, $\MAT I$ is the identity matrix, and $\MAT 0$ is the matrix of zeros with appropriate dimensions. $\MAT K$ is the commutation matrix, which transforms $\vvec(\MAT A)$ into $\vvec( \MAT{A}^\top )$.
Let $\mathcal{S}_d$ and $\mathcal{S}_d^{++}$ denote the set of $d\times d$ real symmetric matrices and the set of $d\times d$ real symmetric positive definite matrices, respectively.
We define the unit sphere on $\mathbb{R}^d$ as $\mathcal{C}_d:=\{\VEC v \in \mathbb{R}^d|\quad \|\VEC v\|_2=1\}$. Similarly, the unit sphere on 
$\mathbb{R}^{p\times n}$ is $\mathcal{C}_{p,n}:=\{ \MAT{U}\in\mathbb{R}^{p\times n}|\ \|\MAT{U}\|_{2} = 1 \}$.
For the statistical notations, $\sim$ means ``distributed as'' while $\stackrel{(d)}{=}$ stands for ``shares the same distribution as''.
$E[\cdot]$ denotes the statistical expectation and $\var[\cdot]$ denotes the statistical variance. For the distributions, $\chi^2_d$ denotes the Chi-squared distribution with $d>0$ degrees of freedom, $\mathcal{F}(d_1,d_2)$ denotes the $F$- distribution with $d_1>0$ and $d_2$ degrees of freedom, and $\Gamma(\alpha,d_1,d_2)$ denotes the generalized Gamma distribution of parameters $\alpha$, $d_1$ and $d_2$.\\
$\Gamma(\cdot)$ is the Gamma function and $\Gamma_p(\cdot)$ is the multivariate Gamma function of dimension $p$.
For $\kappa=(k_1,\dots,k_p) \in \mathbb{N}^p$, we denote by $(\cdot)_\kappa$ the partitional shifted factorial defined for $a \in \mathbb{R}$ at $\kappa$ as $(a)_\kappa = \prod_{j=1}^p \left(a-\frac{j-1}{2}\right)_{k_j}$ where $(a)_{k_j}=a(a+1)\dots(a+k_j-1)$, with the convention $(a)_0=1$. Finally, we introduce the ratio Gamma function:
\begin{equation}
    \gamma(a,b) = \Gamma(a+b)/\Gamma(a).
\end{equation}

\subsection{Statistical properties reminders}
\label{subsec:recall}


This paper focuses on distributions of random matrices over $\mathbb{R}^{p\times n}$. This section provides basic statistical notions and notations needed to manipulate these matrix-variate distributions. As for any distribution, we are interested in the probability density functions, the characteristic functions, and the moments.

The probability density function (PDF) of a random matrix can be defined
like the ones of a random uni-variate variable or random vector.
If $\MAT X \in \mathbb{R}^{p\times n}$  is a random matrix, the matrix-variate PDF of $\MAT X$, denoted as $f$, is a mapping from $\mathbb{R}^{p\times n}$ to $\mathbb{R}_+$ with some parameters describing some characteristics of that distribution.
To be a proper PDF, $f$ must verify
\begin{equation}
    \int_{\mathbb{R}^{p\times n} }f(\MAT X)d\MAT X = 1,
\end{equation}
where $d\MAT X = \prod_{k,\ell}d\MAT X_{k\ell}$ is the Lebesgue measure on $\mathbb{R}^{p\times n}$.

To study a distribution, we systematically derive its moments as they are paramount in characterizing the mean, dispersion around the mean, skewness, and tail's heaviness.
The definition of moments in multivariate statistics was easily extended to the matrix variate case.
In particular, the first moment of a random matrix $\MAT X$ is its expectation $E[\MAT X]$, which is the $p\times n$ matrix whose elements are
\begin{equation}
    E[\MAT X]_{k\ell} = E[\MAT X_{k\ell}] = \int_{\mathbb{R}^{p\times n}} \MAT X_{k\ell} f(\MAT X)d\MAT X.
\end{equation}
The second moment is $E[\vvec(\MAT X)\vvec(\MAT X)^\top]$. Hence, the variance of the random matrix $\MAT X$, denoted $\var[\MAT X]$, is given~by
\begin{equation}
    \var[\MAT X] = E[\vvec(\MAT X)\vvec(\MAT X)^\top]-\vvec(E[\MAT X])\vvec(E[\MAT X])^\top.
\end{equation} 
More generally, it is possible to define the $k^{\textup{th}}$ moment (if it exists) for random matrices to improve our understanding of the distribution shape\footnote{
Its definition can vary from one author to another, but it is more likely given in literature by $E[\MAT X_{j_0j_1}\dots \MAT X_{j_{k-1}j_k}]$ \cite{kollo2005advanced}, the expectation of the product of $k$ elements of $\MAT X$, which is somehow equivalent to the expectation of a Kronecker product of order $k$.}.
For a random matrix $\MAT S \in \mathcal{S}_p$, the $k^{\textup{th}}$ Kronecker moment corresponds to the expectation of $k^{\textup{th}}$ Kronecker product, \textit{i.e.}, $E[\otimes^k\MAT S] = E[\MAT S\otimes \dots \otimes \MAT S] \in \mathbb{R}^{p^k\times p^k}$.
Thanks to these moments, it is possible to compute other types of moments using some vectorization and Kronecker-product tricks.
Examples of such moments are $E[\MAT S^k]$ and $E[\tr(\MAT S)^q\MAT S^r]$ where $q,r \in \mathbb{N}$ \cite{von1988moments}.
One of the available tools to compute Kronecker moments is the differentiation of the characteristic function, which is defined for a random matrix $\MAT X$ as
the Fourier transform of $f$.
The characteristic function $\Psi:\mathbb{R}^{p\times n}\to\mathbb{R}$ is defined as 
\begin{equation}
\begin{split}
    \Psi(\MAT Z)&= E[\exp(i \tr(\MAT X^\top\MAT Z))] 
    \end{split}
\end{equation}
where $i^2=-1$.
If $n=1$, we regain the definition of the characteristic function for the multivariate case.
For symmetric random matrices, Kronecker moments $E[\otimes^k\MAT S]$ are related to the $k^{\textup{th}}$ derivative%
\footnote{
    Let $\MAT X \in \mathbb{R}^{p\times p}$ symmetric and $\MAT Y\in \mathbb{R}^{q\times r}$.
    We choose the Kroneckerian arrangement for defining our matrix derivative $\frac{\partial \MAT Y}{\partial \MAT X} = \sum_{m,j,k,\ell} \alpha_{k\ell}\frac{\partial \MAT Y_{mj}}{\partial \MAT X_{k\ell}}(\VEC s_m\VEC t_j^\top)\otimes (\VEC e_k\VEC e_\ell^\top)$ with $\alpha_{k\ell}=1$ if $k=\ell$, otherwise $1/2$; where $\{\VEC s_m\}_{m=1}^q$, $\{\VEC t_j\}_{j=1}^r$ and $\{\VEC e_k\}_{k=1}^p$ denote the canonical bases of $\mathbb{R}^q$, $\mathbb{R}^r$ and $\mathbb{R}^p$, respectively.
    For more details, see \cite{neudecker2003two,von1988moments}.}
of the characteristic function $\Psi $ with respect to $\MAT Z\in \mathcal{S}_p$ at $\MAT 0$ (if it exists): 
\begin{equation}
    \left.{\frac{\partial^k\Psi (\MAT Z)}{\partial\MAT Z^k}}\right\rvert_{\MAT Z \to \MAT 0} = i^k E[\otimes^k \MAT S].
\label{eq:char_fun_diff}
\end{equation}
\begin{remark}
    We highlight that the second Kronecker moment, $E[\MAT S\otimes \MAT S]$, is different from the second moment, $E[\vvec(\MAT S)\vvec(\MAT S)^\top]$. 
    However, they can be linked to each other: $ \vvec(E[\MAT S\otimes \MAT S]) = (\MAT I_p\otimes \MAT K_{p,p}\otimes \MAT I_p) \vvec(E[\vvec(\MAT S)\vvec(\MAT S)^\top]),$ using the formula $\vvec(\MAT A \otimes \MAT B) = (\MAT I_{n_1}\otimes \MAT K_{n_2,p_1}\otimes \MAT I_{p_2})\left(\vvec(\MAT A)\otimes\vvec(\MAT B)\right)$ for any $\MAT A \in \mathbb{R}^{p_1\times n_1}$ and $\MAT B\in \mathbb{R}^{p_2\times n_2}$, and $\VEC a \otimes \VEC b = \vvec(\VEC b\VEC a^\top )$ for any vectors $\VEC a$ and $\VEC b$~\cite{neudecker1983some}.
\end{remark}

\subsection{Multivariate and matrix-variate elliptical distributions}
\label{subsec:multivariateEllip}

This subsection provides a short overview of the multivariate and matrix-variate elliptical distributions.
These are leveraged to define and study the Elliptical Wishart and Inverse Wishart distributions in the following sections.
We first define the multivariate elliptical distribution in Definition~\ref{def:elliptical_dist}.
For a full review, the reader is referred to~\cite{ollila2012complex,frahm2004generalized}.
\begin{definition}[Multivariate elliptical distribution]
\label{def:elliptical_dist}
    A random vector $\VEC x \in \mathbb{R}^d$ is said to have a centered elliptical distribution, denoted as $\mathcal{E}_d(\VEC 0,\MAT\Sigma,h_d)$ if its probability density function has the form
    \begin{equation}
        f(\VEC x) = |\MAT\Sigma|^{-1/2} h_d(\VEC x^\top\MAT\Sigma^{-1}\VEC x),
    \end{equation}
    where $\MAT\Sigma \in \mathcal{S}_d^{++}$ is the covariance matrix and $h_d:\mathbb{R}_+ \to \mathbb{R}$ is the density generator function of the distribution.
    To ensure that $h_d$ is a proper density generator function, one must have $\frac{\pi^{d/2}}{\Gamma(d/2)}\int_0^{+\infty}h_d(t)t^{\frac{d}{2}-1}dt = 1$.
\end{definition}

\begin{table*}[htb!]
    \centering
    \footnotesize
    \begin{tabular}{c|c|c|c|}
        \cline{2-4}
         & $\mathcal{N}_d(\VEC 0 ,\MAT \Sigma)$ & $t_d(\VEC 0,\MAT \Sigma,\nu)$  & $\mathcal{K}_d(\VEC 0,\MAT \Sigma,\alpha,\beta,R)$
        \\ \hline 
        \multicolumn{1}{|l|}{$h_d(t)$} & $(2\pi)^{-\frac{d}2}\exp\left(-\frac{t}{2}\right)$ & $(\pi\nu)^{-d/2} \frac{\Gamma(\frac{\nu+d}{2})}{\Gamma(\frac{\nu}{2})} \left(1+\frac{t}{\nu}\right)^{-\frac{\nu+d}{2}}$ & $\beta \ \pi^{-d/2}\ R^{\frac{d+2(\alpha-1)}{2\beta}}  \frac{\Gamma(\frac{d}{2})}{\Gamma(\frac{d}{2\beta}+\frac{\alpha-1}{\beta})}\ t^{\alpha-1} \exp(-R t^\beta)$
        \\[4pt] \hline
        \multicolumn{1}{|l|}{$\mathcal{Q}$} & $\mathcal{Q}\sim \chi^2_{d}$ & $\mathcal{Q}\sim d\times\mathcal{F}(d,\nu)$,~\cite{patnaik1949non} & $\mathcal{Q}\sim \Gamma\left(R^{-1/\beta},\alpha-1+\frac{d}{2},\beta\right)$
        \\[4pt] \hline
        \multicolumn{1}{|l|}{$m_k^{(h_d)}$} & $2^k \gamma(\frac{d}2,k)$ & $\nu^{k}\gamma(\frac{d}2,k) \gamma(\frac{\nu}2,-k), \quad -\frac{d}{2}<k<\frac{\nu}{2}$ & $R^{-\frac{k}{\beta}} \gamma(\frac{d}{2\beta}+\frac{\alpha-1}{\beta},\frac{k}{\beta})$, $k>-\frac{d}{2}-\alpha+1$, ~\cite{stacy1962generalization} 
        \\[4pt] \hline
    \end{tabular}
    \caption{Generator density functions $h_d$, the distribution of the second-order modular variable $\mathcal{Q}$~\eqref{eq:storepvectors}, and  moments $m_k^{(h_d)}$~\eqref{eq:moment_modular} of the multivariate Gaussian, $t$- (with $\nu$ as shape parameter), and Kotz (with parameters $\alpha$, $\beta$ and $R$) distributions.}
        \label{table:multivariate_fim}
\end{table*}

The elliptical model can be naturally extended to matrices via the vectorization operator. This family of distributions plays a key role in the definition of Elliptical Wishart distributions. The matrix-variate elliptical distribution is defined in Definition~\ref{def:matrix_elliptical_dist}. For a thorough review, see~\cite{gupta2012elliptically}.

\begin{definition}[Matrix-variate elliptical distributions]
\label{def:matrix_elliptical_dist}
    A random matrix $\MAT X \in \mathbb{R}^{p\times n}$ is said to be drawn from a centered elliptical distribution, denoted as $\mathcal{E}_{p,n}(\MAT 0, \MAT \Sigma, \MAT \Psi,h_{np})$ if $\vvec(\MAT X) \sim \mathcal{E}_{np}(\VEC 0, \MAT \Psi \otimes \MAT \Sigma,h_{np})$ wherein $\MAT \Sigma \in \mathcal{S}_p^{++}$ is the covariance matrix controlling the correlation between $\MAT X$'s rows and $\MAT \Psi \in \mathcal{S}_n^{++}$ is the covariance matrix controlling the correlation between $\MAT X$'s columns. 
\end{definition}

\begin{remark}
    To interpret the two covariance matrices in the scope of signal processing, $\MAT X$ can be seen as a signal matrix of recordings on $p$ channels during $n$ time samples, \emph{i.e.}, the $i$-th column of $\MAT X$ (denoted by $\VEC x_i$) refers the measurements recorded by the $i$-th instant while the $t$-th row of $\MAT X$ (denoted by $\VEC x_{(t)}$) contains all the measurements recorded by the $t$-th channel. Thus, $\MAT \Sigma$ represents the "spatial'' covariance matrix while $\MAT \Psi$ represents the "temporal'' covariance matrix for the recordings. Formally, it can be deduced from~\cite{gupta1994moments}:
\begin{equation}
    \cov[\VEC x_i,\VEC x_j]= \frac{m_1^{(h_{np})}}{np} \MAT\Psi_{ij} \MAT \Sigma; \qquad  \cov[\VEC x_{(t)},\VEC x_{(s)}]=\frac{m_1^{(h_{np})}}{np} \MAT\Sigma_{ts} \MAT \Psi,
    \label{eq:cov_X}
\end{equation}
with $m_1^{(h_{np})}$ is provided by Equation \eqref{eq:moment_modular}.
\end{remark}

Multivariate and matrix-variate elliptical distributions have been well-studied in the literature. In particular, many of their statistical properties, such as the expectation and the variance, were derived in both cases; see \emph{e.g.},~\cite{ollila2012complex,gupta2012elliptically}.
In this paper, we need two tools from the multivariate case: the stochastic representation and moments of an elliptical vector's squared $\ell_2$ norm.
They indeed play a major role in Section~\ref{sec:statistics}.
The stochastic representation of $\VEC x \sim \mathcal{E}_d(\VEC 0,\MAT\Sigma,h_d)$ is
\begin{equation}
\label{eq:storepvectors}
    \VEC x \stackrel{(d)}{=} \sqrt{\mathcal{Q}} \MAT\Sigma^{\frac{1}{2}} \VEC u,
\end{equation}
where $\mathcal{Q}=\|\MAT\Sigma^{-1/2}\VEC x\|_2^2$ is a non-negative random variable and $\VEC u \in \mathbb{R}^d$ is uniformly distributed over the unit sphere $\mathcal{C}_{d}$.
Moreover, $\mathcal{Q}$ and $\VEC u$ are independent.
The variable $\mathcal{Q}$ is called the second-order modular of $\VEC x$ and its  probability density function is $t\mapsto \frac{\pi^{d/2}}{\Gamma(d/2)}h_d(t)t^{\frac{d}{2}-1}$~\cite{cambanis1981theory}.
Furthermore, when $\VEC x \sim \mathcal{E}_{d}(0,\MAT I_{d},h_{d})$, the expectation $m_k^{(h_d)}$ of $\|\VEC x\|_2^{2k}$ is: 
\begin{equation}
    m_k^{(h_d)} = E[\|\VEC x\|_2^{2k}]=E[\mathcal{Q}^k] = \frac{\pi^{d/2}}{\Gamma(d/2)}\int_0^{+\infty}h_d(t)t^{\frac{d}{2}+k-1}dt ,
    \label{eq:moment_modular}
\end{equation}
for $k \in \mathbb{Z}$.
However, this moment does not necessarily exist for all orders. In particular, the convergence of the integral should be checked at $0$ and $+\infty$.

    
At the end of this section, we provide some examples of multivariate elliptical distributions that are well-known in the literature. For instance, we deal with the Gaussian, $t$-, Exponential Power (also called Generalized Gaussian), and Kotz distributions~\cite{muirhead2009aspects,fang2018symmetric}. For each case, we give the generator density function $h_d$,  the distribution of the corresponding $\mathcal{Q}$, and the explicit expression of the moment $m_k^{(h_d)}$ in Table~\ref{table:multivariate_fim}. We would mention that the $t-$ distribution is characterized by its degree of freedom $\nu>0$. It is also called "shape parameter'' as it controls the heaviness of tails. Notice that we approach the normal distribution when $\nu \to \infty$. Conversely, the smaller $\nu$ is, the heavier the tails are. The Exponential Power distribution is also parameterized by a shape parameter $\beta>0$. When $0<\beta<1$, it features heavier tail regions than the normal distribution and hence can be useful in providing robustness against outliers. The
Kotz distribution, denoted by $\mathcal{K}_d(\VEC 0,\MAT\Sigma,\alpha,\beta,R)$, is a generalization of Gaussian and Exponential Power distributions. It has 3 parameters $\alpha>1-\frac{d}2$, $\beta>0$ and $R>0$.
When $\alpha=1$ and $R=\frac{1}{2}$, we get the Exponential Power with parameter $\beta$.
If, in addition, $\beta=1$, we obtain the Gaussian distribution. 

\section{Background: Wishart, Inverse Wishart, and elliptical extensions}
\label{sec:review}

This section aims to provide some background about the Wishart and Inverse Wishart distributions, along with their elliptical extensions. In particular, their definitions, PDFs, and some of their existing statistical properties -- mainly the characteristic functions, the first and second moments, and the Kronecker moments. Notice that in this paper, only non-degenerate distributions are considered, \emph{i.e.}, it is always assumed that the number of samples $n$ is greater than the number of features $p$ and center matrices are non-singular.

\subsection{Wishart distribution}
\label{subsec:w review}

First, we introduce the Wishart distribution; see \emph{e.g.},~\cite{bilodeau1999theory}.
Wishart random matrices are characterized in Definition~\ref{def:defW}.

\begin{definition}[Wishart distribution]
    Let $n$ IID vectors $\{\VEC x_j\}_{1\leq j\leq n}$ in $\mathbb{R}^p$ drawn from the multivariate Gaussian distribution $\mathcal{N}_p(\VEC 0,\MAT\Sigma)$ and  $\MAT S=\sum_{j=1}^n \VEC x_j\VEC x_j^\top$.
    Then, $\MAT S$ is said to have the Wishart distribution with $n$ degrees of freedom and a center matrix $\MAT\Sigma$, denoted as $\mathcal{W}(n,\MAT\Sigma)$.
\label{def:defW}
\end{definition} 
 
\begin{remark}
   Notice that if we consider $\MAT X \in \mathbb{R}^{p\times n}$ the matrix whose columns are the $\VEC x_j$'s, then $\MAT X$ has the matrix-variate Gaussian distribution $\mathcal{N}_{p,n}(\VEC 0,\MAT\Sigma, \MAT I_n)$ \cite{bilodeau1999theory}. We can write $\MAT S =\MAT X \MAT X^\top$, which yields an alternative way to define the Wishart distribution \cite{gupta2018matrix}. Such a matrix formulation is at the core of elliptical extensions, as it will be clarified later.
\end{remark}

The PDF of $\MAT S \sim \mathcal{W}(n,\MAT\Sigma)$ is given by
\begin{equation}
    f (\MAT S) = \frac{1}{2^{\frac{np}2} \Gamma_p(\frac{n}2)} |\MAT\Sigma|^{-\frac{n}2} |\MAT S|^{\frac{n-p-1}{2}} \exp\left(-\frac{1}{2} \tr(\MAT\Sigma^{-1}\MAT S)\right).
    \label{eq:pdfW}
\end{equation}
Various statistical properties of the Wishart distribution are derived thanks to the independence of the $\VEC x_j$'s and the statistical properties of the multivariate Gaussian distribution.
For instance, from~\cite{bilodeau1999theory}, the characteristic function of $\MAT S \sim \mathcal{W}(n,\MAT\Sigma)$ is, for $\MAT Z \in \mathcal{S}_p$,
\begin{equation}
   \Psi (\MAT Z)= |\MAT I_p-2i\MAT Z\MAT\Sigma|^{-\frac{n}{2}}.
   \label{eq:charW}
\end{equation}
The first and second moments of the Wishart distribution are well known~\cite{bilodeau1999theory}. Indeed, its expectation and its variance are 
\begin{equation}
    E[\MAT S]= n \MAT\Sigma,
    \label{eq:expecW}
\end{equation}
and
\begin{equation}
    \var[\MAT S]= n (\MAT I_{p^2}+\MAT K_{p,p}) (\MAT \Sigma \otimes \MAT \Sigma).
    \label{eq:varW}
\end{equation}
The Kronecker moments were derived in \cite{maiwald2000calculation, sultan1996moments} up to order $4$ and were obtained through~\eqref{eq:char_fun_diff}, \emph{i.e.}, the differentiation of the characteristic function~\eqref{eq:charW} of the Wishart distribution.
In particular, \cite{sultan1996moments} provided the explicit expressions of the Kronecker moments of orders $2$, $3$ and $4$. 
To the best of our knowledge, no explicit formula was derived for Kronecker moments of arbitrary order for the Wishart distributions.
However, closely linked to Kronecker moments, $E[\MAT S_{12}\MAT S_{34}\dots\MAT S_{2k-1,2k}]$ were shown to be expressed as weighted generating functions of graphs associated with the Wishart distribution in~\cite{numata2010formulas}, but no closed form was provided.
In the literature, several other types of moments of random matrices drawn from the Wishart distribution can also be found.
For instance, combinatorial formulas for the moments of arbitrary polynomials in the entries of a real Wishart-distributed matrix are obtained in~\cite{lu2001macmahon}.
%
In~\cite{letac2004all, bishop2018introduction}, invariant moments of the form $E[Q(\MAT S)]$, where $Q(\MAT S)$ is a polynomial that only depends on the eigenvalues of $\MAT S$, are computed.
%
Moreover, \cite{graczyk2007moments} derived  $E[\tr(\MAT S^q)\MAT S^r]$ by leveraging the Wishart process and Itô calculus.
Recently, moments $E[\MAT S^k]$ were computed recursively, with explicit formulas up to order $4$~\cite{hillier2021moments}.
In the meantime, \cite{nagar2023expected} gave explicit expressions for $E[(\tr(\MAT S^q))^r(\tr(\MAT S^m))^j]$ for $m+j+q+r\leq 5$.
Finally, one of the invariant moments that has gained a lot of interest in literature is the expectation of zonal polynomials
\footnote{
    In particular, it offers us an alternative way to regain the characteristic function of the Elliptical Wishart distribution.
} of $\MAT S$. 
We recall that a zonal polynomial of $\MAT S$ at any order $\kappa \in \mathbb{N}^p$, denoted as $C_\kappa(\MAT S)$, where $\kappa=(k_1,\dots,k_p)$ and $ k_1\geq \dots \geq k_p$, is a symmetric, homogeneous polynomial of degree $k:=\sum_{j=1}^pk_j$ of the eigenvalues of $\MAT S$; see \emph{e.g.}, \cite{muirhead2009aspects} for a detailed definition.
It is known that, for any deterministic $\MAT Z\in \mathcal{S}_p$,~\cite{khatri1966certain}
\begin{equation}
    E[C_\kappa(\MAT Z\MAT S)] = 2^k \left(\frac{n}{2}\right)_\kappa  C_\kappa(\MAT Z \MAT\Sigma).
\label{eq:zonalW}
\end{equation}

\subsection{Inverse Wishart distribution}
\label{subsec:iw review}

From the Wishart distribution, it is possible to define another major matrix-variate distribution over $\mathcal{S}^{++}_p$: the Inverse Wishart distribution.
%
\begin{definition}
We say that $\MAT S$ has the Inverse Wishart distribution with $n$ degrees of freedom and covariance matrix $\MAT\Sigma$, denoted $\mathcal{W}^{-1}(n,\MAT\Sigma)$, if $\MAT S^{-1}\sim \mathcal{W}(n,\MAT\Sigma^{-1})$.
\label{def:defIW}
\end{definition}

Using the inverse substitution in the integral in \eqref{eq:pdfW}, the probability density function of  $\mathcal{W}^{-1}(n,\Sigma)$ is given by
\begin{equation}
    f (\MAT S) = \frac{1}{2^{\frac{np}2} \Gamma_p(\frac{n}2)} |\MAT\Sigma|^{\frac{n}2} |\MAT S|^{-\frac{n+p+1}{2}} \exp\left(-\frac{1}{2} tr(\MAT\Sigma \MAT S^{-1})\right).
    \label{eq:pdfIW}
\end{equation}
Unlike the Wishart distribution, no closed-form expression is found for the characteristic function of the Inverse Wishart distribution.
The first and second moments need additional assumptions on $n$ and $p$ to exist \cite{gupta2018matrix}.
When $n>p+1$, the expectation of $\MAT S \sim \mathcal{W}^{-1}(n,\MAT\Sigma)$ is
    \begin{equation}
        E[\MAT S] = \frac{1}{n-p-1}\MAT\Sigma,
        \label{eq:expecIW}
    \end{equation}
and when $n>p+3$, the variance exists and is given by
\begin{equation}
    \var[\MAT S] = \frac{1}{(n-p)(n-p-1)(n-p-3)}
    \left[(\MAT I_{p^2}+\MAT K_{p,p})(\MAT\Sigma\otimes \MAT \Sigma) +\frac{2}{n-p-1} \vvec(\MAT \Sigma)\vvec(\MAT\Sigma)^\top\right].
    \label{eq:varIW}
\end{equation}
Interestingly, in contrast to Wishart, the Kronecker moments of the Inverse Wishart random matrix are well known up to any order~\cite{von1988moments} thanks to Stokes' theorem.
Given $k$ such that $n>p+2k+1$, the $(k+1)^{\textup{th}}$ Kronecker moment is given by
\begin{equation}
    \vvec(E[\otimes^{k+1}\MAT S]) = \frac{1}{n-p-1}
    \prod_{i=0}^{k-1}(\MAT I_{p^{2i}}\otimes \MAT A(k-i))(\otimes^{k+1}\vvec(\MAT \Sigma)),
    \label{eq:KronIW}
\end{equation}
where $\MAT A(k)=\left[(n-p-1)\MAT I_{p^{2k+2}}-\sum_{j=0}^{k-1}(\MAT P_1(j)+\MAT P_2(j))\right]^{-1}(\MAT I_{p^k}\otimes \MAT K_{p,p^k}\otimes I_p)\MAT K_{p^{2k},p^2}$ with $\MAT P_1(j)=\left(\otimes^2 (\MAT I_{p^j}\otimes \MAT K_{p,p^{k-1-j}}\otimes \MAT I_p)\right)(\MAT I_{p^k}\otimes \MAT K_{p,p^{k}}\otimes \MAT I_p)(\MAT I_{p^{k-1}}\otimes\MAT K_{p^{k+2},p}) $ and $\MAT P_2(j)=\left(\otimes^2 (\MAT I_{p^j}\otimes \MAT K_{p,p^{k-1-j}}\otimes \MAT I_p)\right)(\MAT I_{p^{2k}}\otimes\MAT K_{p,p})$.
Further notice that moments $E[\otimes ^k \vvec(\MAT S)]$, which are different from our definition of Kronecker moments, were derived in~\cite{von1997moments}.

As for the Wishart distribution, various moments were derived for the Inverse Wishart distribution, such as the expectation of zonal polynomials; see~\cite{khatri1966certain}.
This allowed to derive the expectation of powers of traces, \emph{i.e.}, $\tr(\MAT S)^k$, powers of determinant, \emph{i.e.}, $|\MAT S|^k$, and powers of sums, \emph{i.e.}, $\tr(\MAT S^k)$ with $\MAT S\sim \mathcal{W}^{-1}(n,\MAT\Sigma)$ 
\cite{hillier2022properties}.
In addition, \cite{kan2022expectations} computed moments of $\MAT S^r P(\MAT S)$ where $P$ is a power-sum symmetric function.
Finally, \cite{nagar2023expected} obtained $E[(\tr(\MAT S^q))^r(\tr(\MAT S^i))^j]$ for $i+j+q+r\leq 5$.

\subsection{Elliptical extensions}
The Wishart distribution can be generalized in various ways.
In this work, we are interested in Elliptical Wishart distributions.
They generalize the Wishart distribution the same way multivariate elliptical distributions generalize the Gaussian one.
These are formally defined in Definition~\ref{def:defEW}.
\begin{definition}
    We say that a random matrix $\MAT S$ follows the Elliptical Wishart distribution $\mathcal{EW}(n,\MAT\Sigma,h_{np})$, with $n$ degrees of freedom, center $\MAT\Sigma$, and density generator $h_{np}:\mathbb{R}_+\to\mathbb{R}$, if $\MAT S=\MAT X\MAT X^\top$ such that $\MAT X \sim \mathcal{E}_{p,n}(0,\MAT\Sigma,\MAT I_n,h_{np})$.
\label{def:defEW}
\end{definition}

Consequently, the same density generators $h_{np}$ as in the multivariate case can be considered to define Elliptical Wishart distributions.
For instance, one can leverage density generators from the $t$-, Exponential Power, Kotz, Bessel models, \emph{etc}.
%
The PDF of $\mathcal{EW}(n,\MAT\Sigma,h_{np})$, which can be obtained through some integration substitution tricks~\cite{anderson1958introduction}, is
\begin{equation}
    f(\MAT S) = \frac{\pi^{\frac{np}2}}{\Gamma_p(\frac{n}2)} |\MAT \Sigma|^{-\frac{n}{2}} |\MAT S|^{\frac{n-p-1}{2}} h_{np}\left( \tr(\MAT\Sigma^{-1}\MAT S)\right).
    \label{eq:pdfEW}
\end{equation}

Few theoretical results about Elliptical Wishart distributions are already known.
One of the few statistical properties that were derived using matrix-argument integration is the characteristic function of $\MAT S \sim \mathcal{EW}(n,\MAT\Sigma,h_{np})$ \cite{bekker2015wishart}.
It represents $\Psi$ only on a neighbourhood of $\MAT 0$ in $\mathcal{S}_p$ whose size is governed by
the moment generating function of $\mathcal{Q}$; for heavy-tailed generators such as
the $t$- one, it does not exist.
\begin{equation}
    \Psi (\MAT Z) = \pi^{np/2} \sum_{k=0}^{+\infty}\frac{i^k}{k!} \left(\frac{\int_0^{+\infty}h_{np}(r)r^{\frac{np}{2}+k-1}dr}{\Gamma\left(\frac{np}{2}+k\right)}\right) 
    \sum_{|\kappa|=k} \left(\frac{n}{2}\right)_\kappa C_\kappa(\MAT Z\MAT \Sigma).
    \label{eq:charEW}
\end{equation}
In Section~\ref{subsec:eW new}, a more straightforward way to derive Equation \eqref{eq:charEW} is provided in Proposition \ref{prop:char new}.
Neither the expectation nor the variance of Elliptical Wishart distributions have been computed in the literature.
One of the significant contributions of this work is to derive them in Section~\ref{subsec:eW new}.
However, some other statistical properties have been established, such as the joint density function of eigenvalues%
\footnote{
As a matter of fact,~\cite[Theorem 3.2.17.]{muirhead2009aspects} provides a formula to compute the joint PDF of eigenvalues of a random SPD matrix knowing its PDF.
},
$E[\tr(\MAT S)^k]$ and $E[|\MAT S|^k]$; see \emph{e.g.}, \cite{bekker2015wishart}.
The expectation of zonal polynomials and the cumulative function have also been derived~\cite{caro2016matrix}.
%

Akin to the definition of the Inverse Wishart distributions, it is possible to define the Inverse Elliptical Wishart definitions.
A proper definition is provided in Definition~\ref{def:defEIW}.
\begin{definition}
    We say that $\MAT S$ has the Inverse Elliptical Wishart distribution $\mathcal{EW}^{-1}(n,\MAT\Sigma,h_{np})$ with $n$ degrees of freedom, center $\MAT\Sigma$ and density generator $h_{np}:\mathbb{R}_+\to\mathbb{R}$, if $\MAT S^{-1}\sim\mathcal{EW}(n,\MAT\Sigma^{-1},h_{np})$.
\label{def:defEIW}
\end{definition}

We would highlight that Definition~\ref{def:defEIW} holds as long as $n\geq p$ and $\MAT \Sigma$ is non-singular.
In fact, in this case, $\MAT X \sim \mathcal{E}_{p,n}(\MAT 0,\MAT \Sigma,\MAT I_n,h_{np})$ has a full rank with probability $1$.
Consequently, $\MAT S=\MAT X\MAT X^\top$ is non-singular with probability $1$.
From~\cite{fang1990generalized}, the PDF of the Inverse Elliptical Wishart distribution is
\begin{equation}
    f (\MAT S) = \frac{\pi^{np/2} }{\Gamma_p(n/2)} |\MAT\Sigma|^{n/2} |\MAT S|^{-\frac{n+p+1}{2}} h_{np}\left( \tr(\MAT \Sigma \MAT S^{-1})\right).
    \label{eq:pdfEIW}
\end{equation}
\begin{remark}
    Considering that the Inverse Elliptical Wishart distribution is the extension of the Inverse Wishart distribution that consists in replacing the Gaussian density generator function in \eqref{eq:pdfIW} with a general density generator $h_{np}$, it is obvious that the ``Inverse Elliptical Wishart'' is the same thing as the ``Elliptical Inverse Wishart''.
\end{remark}
For the Inverse Elliptical Wishart distribution, apart from deriving the probability density function~\eqref{eq:pdfEIW} of $\MAT S\sim\mathcal{EW}^{-1}(n,\MAT\Sigma,h_{np})$, no other statistical properties have been derived yet.
In this paper, more precisely in Section~\ref{subsec:eiw new}, the expectation, variance, and Kronecker moments of an Inverse-Elliptical-Wishart-distributed matrix are derived.

To conclude this section, we provide in Table~\ref{tab:recap} a summary of the main statistical properties known before this paper of the Wishart, Inverse Wishart, Elliptical Wishart, and Inverse Elliptical Wishart distributions.

With this paper, Table~\ref{tab:recap} changes quite a lot.
The expectation, variance, and Kronecker moments from any order of Elliptical Wishart and Inverse Elliptical Wishart distributions are derived.
In particular, the Kronecker moments of the Wishart distribution, which were only known up to order 4, are derived up to any order.
Nevertheless, only the characteristic functions of the Inverse Wishart and Inverse Elliptical Wishart distributions remain unknown.

\begin{table}[t!]
\centering
\begin{tabular}{|c|c|c|c|c|c|}
\hline
Distribution & PDF & \begin{tabular}[c]{@{}c@{}}characteristic\\  function \end{tabular} & \begin{tabular}[c]{@{}c@{}c}Expec- \\ tation
\end{tabular}& Variance& \begin{tabular}[c]{@{}c@{}c}Kronecker\\ moments\end{tabular}\tabularnewline \hline
Wishart
&\eqref{eq:pdfW}  &\eqref{eq:charW}  &\eqref{eq:expecW}  &\eqref{eq:varW}
&\begin{tabular}[c]{@{}c@{}c@{}c}up to\\ order 4\\  \cite{sultan1996moments}\end{tabular}    \\ \hline
\begin{tabular}[c]{@{}c@{}}Inverse \\Wishart
\end{tabular} &\eqref{eq:pdfIW}  &\begin{tabular}[c]{@{}c@{}c}No closed\\  form\end{tabular}&\eqref{eq:expecIW}  &\eqref{eq:varIW} &\eqref{eq:KronIW}  \\ \hline
\begin{tabular}[c]{@{}c@{}}Elliptical\\ Wishart
\end{tabular} &\eqref{eq:pdfEW}  &\eqref{eq:charEW}  &\begin{tabular}[c]{@{}c@{}c@{}c}Not\\ derived\\yet\end{tabular}   &\begin{tabular}[c]{@{}c@{}c@{}c}Not\\ derived\\yet\end{tabular} &\begin{tabular}[c]{@{}c@{}c@{}c}Not\\derived\\ yet\end{tabular} \\ \hline
\begin{tabular}[c]{@{}c@{}c@{}c}Inverse\\ Elliptical\\ Wishart
\end{tabular} &\eqref{eq:pdfEIW}  &\begin{tabular}[c]{@{}c@{}c@{}c}Not\\derived\\ yet\end{tabular}  &\begin{tabular}[c]{@{}c@{}c@{}c}Not\\ derived\\yet\end{tabular}   &\begin{tabular}[c]{@{}c@{}c@{}c}Not\\ derived\\yet\end{tabular} &\begin{tabular}[c]{@{}c@{}c@{}c}Not\\derived\\ yet\end{tabular} \\ \hline
\end{tabular}
\caption{Recapitulating table of the \underline{background} about Wishart, Inverse Wishart, Elliptical Wishart, and Inverse Elliptical Wishart distributions.}
\label{tab:recap}
\end{table}

\section{Missing statistical properties of Elliptical Wishart and Inverse Elliptical Wishart}
\label{sec:statistics}

This section contains the main methodological contributions of this paper.
In particular, it provides the expectation and variance of the Elliptical Wishart and Inverse Elliptical Wishart distributions.
Concerning the characteristic function of the Elliptical Wishart distribution, an alternative (and simpler) proof is obtained as compared to the one given in \cite{bekker2015wishart} while the characteristic function of the Inverse Elliptical Wishart remains without a closed form.
All these new results stem from the stochastic representation of an Elliptical-Wishart-distributed matrix, given in Theorem \ref{thm:stoc rep}. 
As for the multivariate elliptical case (see Section \ref{sec:review}), this representation relies on the second-order modular, $\mathcal{Q}$, and a random matrix uniformly distributed on the unit sphere.
\begin{theorem}[Stochastic Representation]
    Let $\MAT S \sim \mathcal{EW}(n,\MAT\Sigma,h_{np})$ where $\MAT\Sigma \in \mathcal{S}_p^{++}$.
    There exist a non negative random variable $\mathcal{Q}$ and a random matrix $\MAT U\in \mathbb{R}^{p\times n}$ such that
    \begin{equation}
        \MAT S \stackrel{(d)}{=} \mathcal{Q} \, \MAT\Sigma^{\frac{1}{2}}(\MAT U\MAT U^\top)\MAT\Sigma^{\frac{1}{2}},\quad \mbox{with}
    \end{equation}
    \begin{itemize}
        \item $\MAT U$ is uniformly distributed on the unit sphere $\mathcal{C}_{p,n}$.
        \item The probability density function of $\mathcal{Q}$ is $t\mapsto \frac{\pi^{np/2}}{\Gamma(np/2)}h_{np}(t)t^{\frac{np}{2}-1}$.
        \item $\MAT U$ and $\mathcal{Q}$ are independent.
    \end{itemize}
\label{thm:stoc rep}
\end{theorem}
\begin{proof}
    By definition, there exists $\MAT X \in \mathbb{R}^{p\times n}$ such that $\MAT S=\MAT X\MAT X^\top$ and $\MAT X \sim \mathcal{E}_{p,n}(0,\MAT\Sigma,\MAT I_n,h_{np})$.
    Recall that $\vvec(\MAT X)\sim\mathcal{E}_{pn}(0,\MAT I_n\otimes \MAT \Sigma,h_{np})$.
    Hence, from Equation~\eqref{eq:storepvectors}, $\vvec(\MAT X)\stackrel{(d)}{=}\sqrt{\mathcal{Q}} (\MAT I_n\otimes\MAT \Sigma)^{\frac{1}{2}}\VEC u$, where $\mathcal{Q}$ is a non negative random variable and $\VEC u$ is uniformly distributed on the $np$-dimensional unit sphere.
    Moreover, $\VEC u$ and $ \mathcal{Q}$ are independent.
    %
    %
    Let $\MAT U\in\mathbb{R}^{p\times n}$ such that $\vvec(\MAT U)=\VEC u$.
    It is easy to check that $\tr(\MAT U \MAT U^\top)=\VEC u^\top \VEC u = 1$.
    Thus, $\MAT U\in\mathcal{C}_{p,n}$.
    Furthermore, $(\MAT I_n\otimes \MAT\Sigma)^{\frac{1}{2}} = \MAT I_n\otimes \MAT\Sigma^{\frac{1}{2}}$.
    It follows that $\vvec(\MAT X)\stackrel{(d)}{=}\sqrt{\mathcal{Q}}(\MAT I_n\otimes \MAT\Sigma^{\frac{1}{2}})\vvec(\MAT U) =\sqrt{\mathcal{Q}}\vvec(\MAT\Sigma^{\frac{1}{2}}\MAT U)$.
    So, $\MAT X\stackrel{(d)}{=}\sqrt{\mathcal{Q}}\MAT\Sigma^{\frac{1}{2}}\MAT U$ and $\MAT S\stackrel{(d)}{=}\mathcal{Q} \MAT\Sigma^{\frac{1}{2}}(\MAT U \MAT U^\top)\MAT \Sigma^{\frac{1}{2}}$.
    This concludes the proof.
\end{proof}

One can recognize that the random matrix $\MAT U\MAT U^\top$ follows the Normalized Wishart distribution~\cite{taylor2017generalization}.
To obtain the statistical properties of the Elliptical Wishart and Inverse Elliptical Wishart distributions, we first need to obtain those of $\MAT U\MAT U^\top$.
This is achieved in Section~\ref{subsec:normW}.
From there, the main results for the Elliptical Wishart and Inverse Elliptical Wishart distributions are given in Sections~\ref{subsec:eW new} and~\ref{subsec:eiw new}, respectively.

\begin{remark} [Rotation invariance]\label{rk:bilodeau}
    Elliptical Wishart and Inverse Elliptical Wishart distributions with $\MAT I_p$ as center matrix along with Normalized Wishart distribution are rotationally-invariant, \emph{i.e.}, $\MAT S\stackrel{(d)}{=}\MAT H\MAT S \MAT H^\top$ for any orthogonal matrix $\MAT H\in\mathbb{R}^{p\times p}$.
    For such invariance property, \cite{tyler1982radial} shows that $E[\MAT S]=\alpha \MAT I_p$ and $\var[\MAT S] = \sigma_1(\MAT I_{p^2}+\MAT K_{p,p})+ \sigma_2 \vvec(\MAT I_p)\vvec(\MAT I_p)^\top$, where $\alpha \in \mathbb{R}$, $\sigma_1 \geq 0$ and $\sigma_2 \geq \frac{-2\sigma_1}{p}$.
    %
    For the distributions we consider, while it enables us to obtain the expectation easily, it does not allow us to identify $\sigma_1$ and $\sigma_2$ straightforwardly.
\end{remark}

\subsection{Normalized Wishart distribution}
\label{subsec:normW}

In this section, the statistical properties of $\MAT U\MAT U^\top$ in the stochastic representation of Theorem~\ref{thm:stoc rep} are studied.
One can observe that, given\footnote{
    As a matter of fact, it works for any Elliptical Wishart distribution, \emph{i.e.}, $\MAT S\sim\mathcal{EW}(n,\MAT I_p,h_{np})$.
} $\MAT S\sim\mathcal{W}(n,\MAT I_p)$,  
 $\MAT U\MAT U^\top \stackrel{(d)}{=}\frac{\MAT S}{\tr(\MAT S)}$.
Thus, as mentioned above, $\MAT U\MAT U^\top$ follows the so-called Normalized Wishart distribution~\cite{taylor2017generalization}. 
In the following, let $\MAT V=\MAT U\MAT U^\top$.
After properly defining the Normalized Wishart distribution in Definition~\ref{def:defNW}, first-order and second-order moments of $\MAT V$ and $\MAT V^{-1}$ are derived in Proposition~\ref{prop:uniform}.
\begin{definition}[Normalized Wishart distribution]
    Let $\MAT S \sim \mathcal{W}(n,\MAT I_p)$ a Wishart matrix.
    The corresponding Normalized Wishart matrix is given by
    \begin{equation}
        \MAT V = \frac{\MAT S}{\tr(\MAT S)}.
    \end{equation}
    Its distribution is denoted $\mathcal{NW}(n,p)$.
\label{def:defNW}
\end{definition}
\begin{proposition}
    Let $\MAT V \sim \mathcal{NW}(n,p)$.
    Then,
    \begin{enumerate}[itemsep=6pt]
        \item $E[\MAT V]=\frac{1}{p}\MAT I_p$,
        \item $\begin{aligned}[t]
        &E[\vvec(\MAT V)\vvec(\MAT V)^\top]
        =\cfrac{1}{np(np+2)} \left(n(\MAT I_{p^2}+\MAT K_{p,p})+n^2\vvec(\MAT I_p)\vvec(\MAT I_p)^\top\right),
        \end{aligned}$
    \item $\MAT V$ is invertible with probability $1$,
        \item $E[\MAT V^{-1}]=\cfrac{np-2}{n-p-1} \, \MAT I_p$,  $\qquad \mbox{if } n>p+1$
        \item $\begin{aligned}[t]
            &E[\vvec(\MAT V^{-1})\vvec(\MAT V^{-1})^\top]=\frac{(np-2)(np-4)}{(n-p)(n-p-1)(n-p-3)}  [\MAT I_{p^2}+\MAT K_{p,p}+(n-p-2)\vvec(\MAT I_p)\vvec(\MAT I_p)^\top],\\
            &\qquad \mbox{ if } n>p+3.
        \end{aligned}$
    \end{enumerate}
\label{prop:uniform}
\end{proposition}
\begin{proof}
    The proof rests on the stochastic representation of Wishart matrices derived in Theorem~\ref{thm:stoc rep} along with the expectation and the variance of the $\chi^2$ distribution and its inverse. Details are provided in Appendix~\ref{app:Normalized Wishart distribution}.
\end{proof}

The stochastic representation of Theorem~\ref{thm:stoc rep} along with Normalized Wishart distribution offers an alternative way to compute the characteristic function of the Elliptical Wishart distribution~\eqref{eq:charEW}.
To obtain it, we thus need the characteristic function of the Normalized Wishart distribution.
It is provided in Proposition~\ref{prop:char unif}.
Before we get to the proposition, the first required thing is the expectation of zonal polynomials of the matrix $\MAT Z\MAT V$.
This is achieved in Lemma~\ref{lem:zonal unif}.
\begin{lemma} 
    Let $\MAT V \sim \mathcal{NW}(n,p)$ and let $\MAT Z\in\mathcal{S}_p$.
    The expectation of the $\kappa$-order zonal polynomial of $\MAT Z \MAT V$ is:
    \begin{equation}
       E[C_\kappa(\MAT Z \MAT V)] = \frac{\Gamma\left(\frac{np}{2}\right)}{\Gamma\left(\frac{np}{2}+k\right)} \left(\frac{n}{2}\right)_\kappa  C_\kappa(\MAT Z).
    \label{eq:zonal unif}
    \end{equation}
\label{lem:zonal unif}
\end{lemma}
\begin{proof}
     The proof rests on the stochastic representation of Wishart matrices derived in Theorem~\ref{thm:stoc rep} and the $k$-th moments of $\chi^2$. Details are provided in Appendix~\ref{app:Normalized Wishart distribution}.
\end{proof}
%
\begin{proposition} 
    Let $\MAT V \sim \mathcal{NW}(n,p)$ .
    The characteristic function of $\MAT V$ at $\MAT Z\in\mathcal{S}_p^+$ is
    \begin{equation}
       \Psi_{\MAT V}(\MAT Z) = \sum_{k=0}^{+\infty}\frac{i^k}{k!}\frac{\Gamma\left(\frac{np}{2}\right)}{\Gamma\left(\frac{np}{2}+k\right)} \sum_{|\kappa|=k}\left(\frac{n}{2}\right)_\kappa  C_\kappa(\MAT Z).
    \label{eq:char unif}
    \end{equation}
\label{prop:char unif}
\end{proposition}
\begin{proof}
    Since $\MAT V \in \mathcal{S}_p^{+}$ and $\tr(\MAT V)=1$, we have $\tr(\MAT Z\MAT V) \leq \|\MAT Z\|$ 
    , and consequently, $E[\tr(\MAT Z \MAT V)^k]\leq \|\MAT Z\|^k$.
    Thus, using the dominated convergence theorem, one can write
    \begin{equation*}
        \Psi_{\MAT V}(\MAT Z)  = E[\exp(i\tr(\MAT Z\MAT V))] = \sum_{k=0}^{+\infty}\frac{i^k}{k!} E[\tr(\MAT Z\MAT V)^k].
    \end{equation*}
    Moreover, from $\tr(\MAT Z\MAT V)^k = \sum_{|\kappa|=k}C_\kappa(\MAT Z\MAT V)$ and from Lemma~\ref{lem:zonal unif}, we get $$E[\tr(\MAT Z\MAT V)^k] = \frac{\Gamma\left(\frac{np}{2}\right)}{\Gamma\left(\frac{np}{2}+k\right)} \sum_{|\kappa|=k}\left(\frac{n}{2}\right)_\kappa  C_\kappa(\MAT Z).$$
    This is enough to conclude.
\end{proof}
\begin{remark}
    One of the applications of deriving the statistical properties of the Normalized Wishart distribution is the study of some Tyler-type estimators of the scatter matrix~\cite{tyler1987distribution}. For instance, \cite{taylor2017generalization}
derived the so-called generalized fixed point estimate (GFPE) for scaled
packets of data. The Normalized Wishart distribution was a primordial asset for authors to study the behavior of their estimator.
\end{remark}

\subsection{Elliptical Wishart distribution}
\label{subsec:eW new}

We now have all the necessary tools to derive the statistical properties of the Elliptical Wishart distribution.
First, leveraging the stochastic representation from Theorem~\ref{thm:stoc rep} along with the expectation and variance of the Normalized Wishart distribution from Proposition~\ref{prop:uniform}, the expectation and variance of the Elliptical Wishart distribution $\mathcal{EW}(n,\MAT \Sigma,h_{np})$ are derived.
This is achieved in Propositions~\ref{prop:expec} and~\ref{prop:variance}.
\begin{proposition}[Expectation]
    Suppose that moment $m_1^{(h_{np})}$ from~\eqref{eq:moment_modular} is finite.
    The expectation of $\MAT S\sim\mathcal{EW}(n,\MAT\Sigma,h_{np})$ is 
    \begin{equation}
        E[\MAT S] = a_{n,p} \MAT\Sigma, 
        \label{eq:esp}
    \end{equation}
\label{prop:expec}
where $a_{n,p}=\frac{m_1^{(h_{np})}}{p}$.
\end{proposition}
\begin{proof}
    From Theorem~\ref{thm:stoc rep}, $\MAT S=\mathcal{Q}\MAT\Sigma^{\frac{1}{2}}\MAT U\MAT U^\top\MAT\Sigma^{\frac{1}{2}}$.
    Since $\MAT U$ and $\mathcal{Q}$ are independent, $$E[\MAT S]= E[\mathcal{Q}]\MAT\Sigma^{\frac{1}{2}}E[\MAT U\MAT U^\top]\MAT\Sigma^{\frac{1}{2}}.$$
    As $\MAT U \MAT U^\top \sim \mathcal{NW}(n,p)$, the point 1. of Proposition~\ref{prop:uniform} yields the result.
\end{proof}
\begin{proposition}[Variance]
    Suppose that $m_1^{(h_{np})}$ and $m_2^{(h_{np})}$ from~\eqref{eq:moment_modular} are finite.
    The variance of $\MAT S\sim\mathcal{EW}(n,\MAT\Sigma,h_{np})$ is
    \begin{equation}
       \var[\MAT S] = n b_{n,p} (\MAT I_{p^2}+\MAT K_{p,p})(\MAT\Sigma\otimes\MAT\Sigma)\\
       + n^2 c_{n,p} \vvec(\MAT\Sigma)\vvec(\MAT\Sigma)^\top,
       \label{eq:var}
    \end{equation}
    where $b_{n,p}=\frac{m_2^{(h_{np})}}{np(np+2)}$ and $c_{n,p}=b_{n,p} -\left(\frac{m_1^{(h_{np})}}{np}\right)^2$.
\label{prop:variance}
\end{proposition}
\begin{proof}
    The proof combines the stochastic representation and Point 2. from Proposition~\ref{prop:uniform}. See Appendix~\ref{app:Elliptical Wishart distribution} for details.
\end{proof}

The stochastic representation from Theorem~\ref{thm:stoc rep} offers an alternative way from~\cite{caro2008noncentral} to derive the characteristic function~\eqref{eq:charEW} of the Elliptical Wishart distribution as a series expansion, derived in Proposition \ref{prop:char new}; the novelty is not in the result itself but in the way of proving it and the exact domain of its validity.
    It is actually simpler because it does not use the matrix-argument integral equality proved in ~\cite[Theorem 2.1.1]{caro2008noncentral}.

\begin{proposition}[Characteristic function]
    \label{prop:char new}
Let $\MAT S\sim\mathcal{EW}(n,\MAT\Sigma,h_{np})$ with second-order modular $\mathcal{Q}$, and let
\begin{equation*}
t_0:=\sup\{t\ge0:\ E[e^{t\mathcal{Q}}]<\infty\}\in[0,+\infty].
\end{equation*}
Equivalently, $t_0$ is the convergence radius of the series $\sum_{k} m_k^{(h_{np})}\frac{t^k}{k!}$.
\begin{enumerate}[label=(\roman*)]
\item For every $\MAT Z\in\mathcal{S}_p$ with $\rho(\MAT Z\MAT\Sigma)<t_0$, the series in Equation \eqref{eq:charEW} converges absolutely and equals $\Psi(\MAT Z)$.
\item For every $\MAT Z \in \mathcal{S}_p$ with $\rho(\MAT Z\MAT\Sigma)>t_0$, the series in Equation \eqref{eq:charEW} diverges.
\end{enumerate}
In particular: $t_0=\frac12$ for the Wishart distribution; for the Kotz Wishart distribution, $t_0=+\infty$ if $\beta>1$, $t_0=R$ if $\beta=1$, $t_0=0$ if $\beta<1$; for the $t$-Wishart distribution, $m_k=+\infty$ for $k\ge\nu/2$ and \eqref{eq:charEW} is undefined.
\end{proposition}
\begin{proof}
(i) Let $\MAT A:=\MAT \Sigma^{1/2}\MAT Z\MAT \Sigma^{1/2}\in\mathcal{S}_p$. Since $\MAT A=\MAT \Sigma^{1/2}(\MAT Z\MAT \Sigma)\MAT \Sigma^{-1/2}$ is symmetric, $\|\MAT A\|=\rho(\MAT Z\MAT \Sigma)$. By Theorem~\ref{thm:stoc rep}, $\tr(\MAT Z\MAT S)\overset d=\mathcal{Q}\tr(\MAT A\MAT V)$ with $\MAT V\sim\mathcal{NW}(n,p)$ independent of $\mathcal{Q}$. As $\MAT V\succeq\MAT 0$ and $\tr(\MAT V)=1$, $|\tr(\MAT A\MAT V)|\le\|\MAT A\|$. Hence
\begin{equation*}
\sum_{k\ge0}\frac{1}{k!}E\big|\mathcal{Q}\tr(\MAT A\MAT V)\big|^k\le E\big[e^{\|\MAT A\|\mathcal{Q}}\big]<\infty ,
\end{equation*}
and Fubini's theorem together with the independence of $\mathcal{Q}$ and $\MAT V$ gives $\Psi(\MAT Z)=\sum_{k\ge0}\frac{i^k}{k!}m_k^{(h_{np})}\,E[\tr(\MAT A\MAT V)^k]$. Lemma~\ref{lem:zonal unif} and $\tr(\MAT A\MAT V)^k=\sum_{|\kappa|=k}C_\kappa(\MAT A\MAT V)$ give $E[\tr(\MAT A\MAT V)^k]=\frac{\Gamma(np/2)}{\Gamma(np/2+k)}\sum_{|\kappa|=k}(\frac n2)_\kappa C_\kappa(\MAT A)$, and $C_\kappa(\MAT A)=C_\kappa(\MAT Z\MAT \Sigma)$ because the two matrices have the same eigenvalues. The $k$-th term is bounded in modulus by $m_k^{(h_{np})}\|\MAT A\|^k/k!$, whence absolute convergence.

(ii) The proof of the series divergence when $\rho(\MAT Z\Sigma)>t_0$ uses the stochastic representation like (i), but is more technical. We provide it in details in Appendix \ref{app:Elliptical Wishart distribution}.
\end{proof}

\begin{remark}
    Interestingly, the derived stochastic representation can pave the way to easily compute other matrix-argument integrals without delving into the heavy computations and with no use of manipulations from the matrix-argument integrals~\cite{muirhead2009aspects}. 
    
\end{remark}

Explicit expressions of $a_{n,p}$, $b_{n,p}$ and $c_{n,p}$ for the Wishart, $t$-Wishart and Kotz Wishart distributions are given in Table~\ref{tab:examples_EW}. The last statistical property of the Elliptical Wishart distribution studied in this work is the Kronecker moments at any order.
Again, we resort to the stochastic representation from Theorem~\ref{thm:stoc rep}.
Hence, to obtain Kronecker moments of the Elliptical Wishart distribution, we rely on Kronecker moments of the Wishart distribution.
As mentioned in Section~\ref{subsec:w review}, Kronecker moments of a Wishart-distributed matrix are derived in~\cite{sultan1996moments} up to order $4$ by recursively differentiating $4$ times the characteristic function of the Wishart distribution.
Our first contribution is to provide a novel recursive way to compute Kronecker moments of the Wishart distribution up to any order, which leads to non-recursive explicit formulas.
This is achieved in Proposition~\ref{prop:kronW}.
From there, Kronecker moments of the Elliptical Wishart distribution are obtained in Proposition~\ref{prop:kronEW}.
\begin{proposition}[Kronecker moments]
    Let $\MAT S \sim \mathcal{W}(n,\MAT \Sigma)$ where $\MAT\Sigma \in \mathcal{S}_p^{++}$.
    Then, for $k \in \mathbb{N}^*$,
    \begin{equation}
       E[\otimes^{k+1}\MAT S] =  n E[\otimes^k \MAT S]\otimes \MAT\Sigma -2 \frac{\partial E[\otimes^{k}\MAT S]}{\partial \MAT \Sigma^{-1}}.
       \label{eq:kronW_rec}
    \end{equation}
    In a non-recursive way,
    \begin{equation}
        \vvec(E[\otimes^{k}\MAT S]) = \left(\otimes^{2k}\MAT \Sigma^{\frac{1}{2}}\right) \prod_{l=0}^{k-1}\left[\MAT I_{p^{2l}}\otimes\MAT M_{(k-1-l)}\right]
        \left(\otimes^k \vvec(\MAT I_p)\right),
        \label{eq:kronW}
    \end{equation}
    where
        $\MAT M_{(k)}  = \left[n (\MAT I_{p^k}\otimes \MAT K_{p,p^k} \otimes \MAT I_p ) +\MAT J_{(k)} \right] \MAT K_{p^{2k},p^2}$;
        $\MAT J_{(k)}  = \left[\sum_{l=0}^{k-1}(\MAT H_{(k,l)}\otimes \MAT H_{(k,l)})\right](\MAT I_{p^{k-1}}\otimes \MAT K_{p^2,p^{k-1}}\otimes \MAT I_{p^2})\left(\MAT I_{p^{2k-2}}\otimes\MAT G\right)(\MAT I_{p^{k-1}} \otimes \MAT K_{p^{k-1},p^2}\otimes \MAT I_{p^2})(\MAT I_{p^{k}} \otimes \MAT K_{p,p^{k}}\otimes \MAT I_p)$; 
        $\MAT H_{(k,l)}  = \left(\MAT I_{p^l}\otimes \MAT K_{p,p^{k-1-l}}\otimes \MAT I_{p}\right)$;
        and
        $\MAT G  = \frac{1}{2}(\MAT K_{pp}\otimes\MAT K_{pp})(\MAT I_p\otimes \MAT K_{pp}\otimes \MAT I_p) \left[(\MAT I_{p^2}+\MAT K_{pp})\otimes (\MAT I_{p^2}+\MAT K_{pp})\right]$.
\label{prop:kronW}
\end{proposition}
\begin{proof}
The proof is technical. It rests on the differentiation under the matrix-argument integral sign with some cumbersome matrix manipulations. Therefore, we choose to place the proof in Appendix~\ref{app:Elliptical Wishart distribution}.
\end{proof}
\begin{proposition}[Kronecker moments]
    Let $\MAT S\sim\mathcal{EW}(n,\MAT \Sigma, h_{np})$.
    Then, for $k\in\mathbb{N}^*$ such that $m_k^{(h_{np})}$ defined in~\eqref{eq:moment_modular} is finite,
    \begin{equation}
        \vvec(E[\otimes^{k}\MAT S]) = \frac{\Gamma(\frac{np}{2})}{2^k\Gamma(\frac{np}{2}+k)} m_k^{(h_{np})} \left(\otimes^{2k}\MAT \Sigma^{\frac{1}{2}}\right)
        \prod_{l=0}^{k-1}\left[\MAT I_{p^{2l}}\otimes\MAT M_{(k-1-l)}\right] \left(\otimes^k \vvec(\MAT I_p)\right),
    \label{eq:kronEW}
    \end{equation}
    where $\MAT M_{(k)}$ is defined in Proposition~\ref{prop:kronW}.
\label{prop:kronEW}
\end{proposition}
\begin{proof}
    The proof uses the stochastic representation of the Elliptical Wishart distribution, given in Theorem~\ref{thm:stoc rep}, and the Kronecker moments of the Wishart distribution derived in Proposition~\ref{prop:kronW}. Details are provided in Appendix~\ref{app:Elliptical Wishart distribution}.
\end{proof}

\subsection{Inverse Elliptical Wishart distribution}
\label{subsec:eiw new}

After studying the Elliptical Wishart distribution, we turn to the Inverse Elliptical Wishart distribution.
As in the previous section, the stochastic representation from Theorem~\ref{thm:stoc rep} along with the statistical properties of the Normalized Wishart distribution from Section~\ref{subsec:normW} are exploited.
First, Proposition~\ref{prop:uniform} is exploited to derive the expectation and variance of the Inverse Elliptical Wishart distribution $\mathcal{EW}^{-1}(n,\MAT \Sigma,h_{np})$.
These are in Propositions~\ref{prop:esp inv} and~\ref{prop:var inv}.
\begin{proposition}[Expectation]
    Suppose that $n>p+1$ and moment $m_{-1}^{(h_{np})}$ from~\eqref{eq:moment_modular} is finite.
    The expectation of $\MAT S\sim\mathcal{EW}^{-1}(n,\MAT\Sigma,h_{np})$ is 
    \begin{equation}
       E[\MAT S] = d_{n,p}\MAT\Sigma,
    \label{eq:esp inv}
    \end{equation}
    where $d_{n,p}=\frac{np-2}{n-p-1} m_{-1}^{(h_{np})} $.
\label{prop:esp inv}
\end{proposition}
\begin{proof}
    The proof combines the stochastic representation and Point 4. from Proposition~\ref{prop:uniform}. Details are provided in Appendix~\ref{app:Inverse Elliptical Wishart distribution}.
\end{proof}
%
\begin{proposition}[Variance]
    Suppose that $n>p+3$ and moments $m_{-1}^{(h_{np})}$ and $m_{-2}^{(h_{np})}$ are finite.
    The variance of $\MAT S\sim\mathcal{EW}^{-1}(n,\MAT\Sigma,h_{np})$ is
    \begin{equation}
        \var[\MAT S] =  e_{n,p} (\MAT I_{p^2}+\MAT K_{p,p})(\MAT\Sigma\otimes\MAT\Sigma)
        + (n-p-2)f_{n,p} \vvec(\MAT\Sigma)\vvec(\MAT\Sigma)^\top,
       \label{eq:var inv}
    \end{equation}
    where $e_{n,p}=\frac{(np-2)(np-4)}{(n-p)(n-p-1)(n-p-3)} m_{-2}^{(h_{np})}$ and $f_{n,p}=e_{n,p} - \frac{1}{n-p-2}\left(\frac{np-2}{n-p-1} m_{-1}^{(h_{np})}\right)^2$.
\label{prop:var inv}
\end{proposition}
\begin{proof}
    The proof combines the stochastic representation and Point 5. of Proposition~\ref{prop:uniform}. Details are given in Appendix~\ref{app:Inverse Elliptical Wishart distribution}.
\end{proof}

Explicit expressions of $d_{n,p}$, $e_{n,p}$ and $f_{n,p}$ for the Wishart, $t$-Wishart and Kotz Wishart distributions are given in Table~\ref{tab:examples_EW}.
As for the Inverse Wishart distribution, the characteristic function of the Inverse Elliptical Wishart distribution appears quite complicated to obtain and remains unknown.
The last statistical property of the Inverse Elliptical Wishart distribution $\mathcal{EW}^{-1}(n,\MAT\Sigma,h_{np})$ that is derived is the Kronecker moments.
These are provided in Proposition~\ref{prop:kronIEW}.
\begin{proposition}[Kronecker moments] 
    Let $\MAT S \sim \mathcal{EW}^{-1}(n,\MAT \Sigma,h_{np})$.
    Let $k \in \mathbb{N}$ such that  $n>p+2k+1$ and suppose that $m_{-k-1}^{(h_{np})}$ defined in~\eqref{eq:moment_modular} is finite. Then,
    \begin{equation}
        \vvec(E[\otimes^{k+1}\MAT S]) =   \frac{2^{k+1}\Gamma(\frac{np}{2})}{(n-p-1)\Gamma(\frac{np}{2}-k-1)}\ m_{-k-1}^{(h_{np})}
        \prod_{i=0}^{k-1}(\MAT I_{p^{2i}}\otimes \MAT A(k-i))(\otimes^{k+1}\vvec(\MAT \Sigma)),
    \end{equation}
    where $\MAT A(k)$ is defined in~\eqref{eq:KronIW}.
\label{prop:kronIEW}
\end{proposition}
\begin{proof}
    The proof leverages the stochastic representation given in Theorem~\ref{thm:stoc rep} and the Kronecker moments of the Inverse Wishart distribution derived in Equation~\eqref{eq:KronIW}. Details are given in Appendix~\ref{app:Inverse Elliptical Wishart distribution}.
\end{proof}
\begin{table*}
    \centering
    \footnotesize
    \begin{tabular}{c|c|c|c|}
    \cline{2-4}
 & $\mathcal{W}(n,\MAT\Sigma)$ & $t\mathcal{W}(n,\MAT\Sigma,\nu)$ & $\mathcal{KW}(n,\MAT\Sigma,\alpha,\beta,R)$ \\ \cline{1-4}
 \multicolumn{1}{|l|}{$a_{n,p}$} &$n$&$\frac{\nu}{\nu-2} n, \text{ for }\nu>2$&$\frac{1}{p}  R^{-\frac{1}{\beta}} \gamma_1$
 \\  \hline
 \multicolumn{1}{|l|}{$b_{n,p}$} & $1$&$\frac{\nu^2}{(\nu-2)(\nu-4)},$ for $\nu>4$&$\frac{1}{np(np+2)}  R^{-\frac{2}{\beta}}\ \gamma_2$
 \\ \hline
 \multicolumn{1}{|l|}{$c_{n,p}$} & $0$&$\frac{2\nu^2}{(\nu-2)^2(\nu-4)}, $ for $\nu>4$&$\frac{1}{np(np+2)}  R^{-\frac{2}{\beta}}[ \gamma_2 - \left(1+\frac{2}{np}\right) \gamma_1^2]$ \\ \hline
 \multicolumn{1}{|l|}{$d_{n,p}$} &$\frac{1}{n-p-1}$&$\frac{1}{n-p-1}$&$\frac{np-2}{n-p-1}R^{\frac{1}{\beta}} \gamma_{-1}$ \\ \hline
 \multicolumn{1}{|l|}{$e_{n,p}$} & $\frac{1}{D}$&$\frac{1}{D}\left(1+\frac{2}{\nu}\right)$&$\frac{(np-2)(np-4)}{D}R^{\frac{2}{\beta}} \gamma_{-2}$ \\ \hline
 \multicolumn{1}{|l|}{$f_{n,p}$}& $\frac{2}{D(n-p-1)(n-p-2)}$&$\frac{2}{D(n-p-1)(n-p-2)}+\frac{2}{D\nu}$&\begin{tabular}[c]{@{}c@{}}
 $\frac{(np-2)(np-4)}{D}R^{\frac{2}{\beta}}\Big[\gamma_{-2}-\frac{(n-p)(n-p-3)(np-2)}{(n-p-1)(n-p-2)(np-4)}\,\gamma_{-1}^2\Big]$
 \end{tabular} \\
 \hline
    \end{tabular}
     \caption{Explicit formulas for $a_{n,p},\ b_{n,p},\ c_{n,p},\ d_{n,p},\ e_{n,p}$, and $f_{n,p}$ given in Proposition~\ref{prop:expec}, \ref{prop:variance}, \ref{prop:esp inv} and \ref{prop:var inv} for the Wishart, $t$-Wishart (with $\nu$ as shape parameter) and Kotz Wishart (with parameters $\alpha$, $\beta$ and $R$) distributions. We introduce $D= (n-p)(n-p-1)(n-p-3)$; $\gamma_k$ denotes $\gamma(\frac{np}{2\beta}+\frac{\alpha-1}{\beta},\frac{k}{\beta})$ for $k \in \mathbb{Z}$.}
         \label{tab:examples_EW}
\end{table*}

\section{Dependence structure of Elliptical Wishart distributions }
\label{sec:compare_scm}
Even if the generalization offered by the Elliptical Wishart distributions of the Wishart distribution is similar to how multivariate elliptical distributions extend the multivariate Gaussian, the formers provide a richer structure. 
First of all, notice that the columns $\{\VEC x_i\}_{i=1}^n$ of $\MAT X\sim \mathcal{E}_{p,n}(\MAT 0, \MAT\Sigma,\MAT I_n,h_{np})$ are  decorrelated since $\cov(\VEC x_i,\VEC x_j)=\MAT 0$ for $i\neq j$ (according to Equation~\eqref{eq:cov_X}). However, they are not independent, except for the Gaussian case\footnote{The only case where the PDF of $\MAT X$ is the product of the PDFs of its columns is when the density generator function verifies the exponentiation identity, \emph{i.e}, when it is the Gaussian density generator function}. This already highlights that the robustness offered by the Elliptical Wishart modeling covers the whole covariance $\MAT X\MAT X^\top$, while the multivariate elliptical distributions bring robustness at the scale of the columns only.
In this section, we emphasize the significant difference
between the Elliptical Wishart distribution and the distribution of the SCM of IID centered elliptical vectors, due to the underlying dependence structure in the Elliptical Wishart matrices.

Let us delve into the details of both variables. Let $\MAT S = \MAT X\MAT X^T$ such that $\MAT X\sim \mathcal{E}_{p,n}(\MAT 0,\MAT\Sigma,\MAT I_n,h_{np})$. By definition, $\MAT S \sim \mathcal{EW}(n,\MAT\Sigma,h_{np})$. 
Besides, one can write $\MAT S = \sum_{j=1}^n\VEC x_j \VEC x_j^\top$. For each $j\in [\![1,n]\!]$, $\VEC x_j$ is a subvector of $\vvec(\MAT X)$, where $\vvec(\MAT X)$ is drawn from $\mathcal{E}_{np}(\VEC 0, \MAT I_n \otimes \MAT\Sigma,h_{np})$. Thus, using \cite{cambanis1981theory}, we can show  that $\VEC x_j \sim \mathcal{E}_p(\VEC 0 ,\MAT \Sigma, g_{np})$ where
\begin{equation}
    g_{np}(t) = \frac{\pi^{p(n-1)/2}}{\Gamma\left(\frac{p(n-1)}{2}\right)}\int_{0}^{+\infty} h_{np}(t+r)r^{\frac{p(n-1)}{2}-1}dr.
\end{equation}
For instance, in the case of $t$-$\mathcal{W}(n,\MAT\Sigma,\nu)$, the columns $\VEC x_i$'s are drawn from $t_p(\MAT \Sigma,\nu)$.

One the other side, let $\{\VEC y_j\}_{j=1}^n$ IID random vectors drawn from the distribution $\mathcal{E}_p(\VEC 0,\MAT\Sigma,g_{np})$. We introduce $\MAT R :=\sum_{j=1}^n \VEC y_j \VEC y_j^\top$, the SCM of the $\VEC y_j$'s, multiplied by $n$. From~\cite{eaton1973non}, as $n\geq p$, $\MAT R$ is non-singular with probability one and thus $\MAT R\in \mathcal{S}_p^{++}$ with probability one. To investigate its distribution and compare it to the distribution of $\MAT S$, we propose to derive the expectation, variance, and characteristic function of $\MAT R$.  To begin with, we point out that, if they exist, $m_k^{(g_{np})}$ and $m_k^{(h_{np})}$ are linked:
\begin{equation}
        \frac{\Gamma(\frac{p}2)}{\Gamma(\frac{p}2+k)}m^{(g_{np})}_k =  \frac{\Gamma(\frac{np}2)}{\Gamma(\frac{np}2+k)}m^{(h_{np})}_k.
        \label{eq:link_moments}
\end{equation}
This can be shown by a simple variable substitution in integrals. Second, we highlight that $\VEC y_j\VEC y_j^\top  \stackrel{IID}{\sim }\mathcal{EW}(1,\MAT\Sigma,g_{np})$. Note that Equations \eqref{eq:esp},\eqref{eq:var}, and \eqref{eq:charEW} are still valid for $\mathcal{EW}(1,\MAT\Sigma,g_{np})$ although the corresponding degree of freedom, equal to $1$, is less than \footnote{Alternatively, we can use results from multivariate elliptical distributions.} $p$. Suppose that $m_2^{(h_{np})}$ is finite. Along with the independence of the $\VEC y_j$'s, this yields the expectation of $\MAT R$ as $E[\MAT R] = n E[\VEC y_1\VEC y_1^\top]$, \emph{i.e.},
\begin{equation}
    E[\MAT R] =  a_{n,p}\MAT\Sigma,
\end{equation}
and its variance $\var[\MAT R] = n \  \var[\VEC y_1\VEC y_1^\top]$, \emph{i.e.},
\begin{equation}
    \var[\MAT R]= n\  b_{n,p}(\MAT I_{p^2}+\MAT K_{pp})(\MAT\Sigma\otimes\MAT \Sigma)+n \ c_{n,p}\vvec(\MAT\Sigma)\vvec(\MAT\Sigma)^\top,
\end{equation}
with $a_{n,p}$, $b_{n,p}$ and $c_{n,p}$ are given in Proposition \ref{prop:expec} and \ref{prop:variance}.
To have a better understanding of the model's distribution, we compute the characteristic function of $\MAT R$ with the assumption that $m^{(g_{np})}_k<+\infty$\footnote{$m^{(g_{np})}_k<+\infty \Longleftrightarrow m^{({h}_{np})}_k<+\infty$, thanks to Equation \eqref{eq:link_moments}} for each integer $k$. In fact, $\Psi_{\MAT R}(\MAT Z)= \Psi_{\VEC y_1\VEC y_1^\top}(\MAT Z)^n$. Thus,
\begin{equation}
\begin{split}
    \Psi_{\MAT R}&(\MAT Z)
    =\left( \sum_{k=0}^{+\infty}\frac{\mathbf{i}^k}{k!} \frac{\Gamma(\frac{p}2)}{\Gamma(\frac{p}2+k)}\ m^{(g_{np})}_k \sum_{|\kappa|=k} \left(\frac{1}{2}\right)_{\kappa} C_\kappa(\MAT Z\MAT\Sigma)\right)^n\\
    &=\left( \sum_{k=0}^{+\infty}\frac{\mathbf{i}^k}{k!} \frac{\Gamma(np/2)}{\Gamma(np/2+k)}\ m^{({h}_{np})}_k \sum_{|\kappa|=k} \left(\frac{1}{2}\right)_{\kappa} C_\kappa(\MAT Z\MAT\Sigma)\right)^n.
    \label{eq:charSCM}
    \end{split}
\end{equation}

\begin{remark}
    In the case where $h_{np}$ is the Gaussian density generator function (and consequently, $g_{np}$ is also the Gaussian density generator), independence gives directly $\Psi_{\MAT R}(\MAT Z)= \Psi_{\MAT y_1 \MAT y_1^\top}(\MAT Z)^n= |\MAT I_p -2i\MAT Z\MAT\Sigma|^{-n/2}$, the Wishart characteristic function \eqref{eq:charW}. This closed form exists for any $\MAT Z \in \mathcal{S}_p$. The series \eqref{eq:charSCM} represents
it for $\rho(\MAT Z\MAT \Sigma)<\frac{1}{2}$ (diverges if $\rho(\MAT Z\MAT \Sigma)>\frac{1}{2}$). In fact, $|\MAT I_p - \MAT A|^{-\alpha} = \sum_{k=0}^{+\infty}\frac{1}{k!}\sum_{|\kappa|=k}\left(\alpha\right)_\kappa C_\kappa(\MAT A) $ for any $\alpha \in \mathbb{C}$ and any complex symmetric $p\times p$ matrix $\MAT A$ with $\rho(\MAT A)<1$ \cite{muirhead2009aspects}, applied for $\MAT A =2i\MAT\Sigma^{1/2}\MAT Z\MAT \Sigma^{1/2}$.
 \end{remark}

We notice that $\MAT R$ and $\MAT S$ share the same expectation. However, when it comes to the variance, the differences are significant except for the Gaussian case where obviously $\MAT R \overset{d}{=} \MAT S$. We can observe that the term in front of $\vvec(\MAT\Sigma)\vvec(\MAT\Sigma)^\top$ is $n \ c_{n,p}$ for $\MAT R$, while it is $n^2c_{n,p}$ for $\MAT S$.  Furthermore, if the series expansion of the characteristic function exists, the difference between Equation \eqref{eq:charEW} and \eqref{eq:charSCM} is substantial. This arises from the dependence structure of $\MAT S$. In particular, such a structure significantly impacts the asymptotic behavior of $\MAT S$, as provided in Section~\ref{sec:asymptotic_elliptical_wishart}. 

\section{Asymptotic behavior of Elliptical Wishart distributions}
\label{sec:asymptotic_elliptical_wishart}

In this section, we investigate the behavior of Elliptical Wishart distributions when the degree of freedom $n$ tends to $+\infty$. 
While $n$ can be seen as the number of observations for
the Wishart matrix, the Elliptical Wishart case is more complicated because $n$ also complexifies the dependence structure between the $\VEC x_i$'s, the columns of $\MAT X$ such that $\MAT S=\MAT X\MAT X^\top$ (Definition \ref{def:defEW}), as detailed in Section \ref{sec:compare_scm}. The dependence between the $\VEC x_i$'s does not allow the application of the Central Theorem Limit (CLT). To overcome this, we suggest first derive the asymptotic behavior of the Normalized Wishart distribution and then deduce the asymptotic behavior of Elliptical Wishart distributions.

Let us recall the asymptotic behavior of the Wishart distribution. Let $\MAT S_n = \sum_{i=1}^n \VEC x_i\VEC x_i^\top$ where $\VEC x_i\overset{IID}{\sim}\mathcal{N}_p(0,\MAT\Sigma)$. In this subsection, we opt for changing the notation $\MAT S$ to $\MAT S_n$ in order to emphasize the asymptotic regime $n\to +\infty$ while $p$ is fixed. By definition, $\MAT S_n \sim \mathcal{W}(n,\MAT\Sigma)$.  
Using CLT applied on $\MAT z_i:=\VEC x_i\VEC x_i^\top$,
\begin{equation}
    \sqrt{n}\ \vvec\left(\frac{\MAT S_n}{n}-\MAT\Sigma\right) \overset{\textup{(d)}}{\underset{n\to+\infty}{\longrightarrow}} \mathcal{N}_{p^2}(0,(\MAT I_{p^2}+\MAT K_{pp})(\MAT \Sigma\otimes \MAT\Sigma)),
    \label{eq:clt_}
\end{equation}
because the $\MAT z_i$'s are IID drawn from $\mathcal{W}(1,\MAT\Sigma)$, in particular, $E[\MAT z_i]=\MAT \Sigma$ and $\var[\MAT z_i]=(\MAT I_{p^2}+\MAT K_{pp})(\MAT \Sigma\otimes \MAT\Sigma)$. Deriving the asymptotic distribution of the Wishart matrix enables us to obtain the asymptotic behavior of the Normalized Wishart distribution, as stated by Proposition \ref{prop:asymp_normalized_wishart}.

\begin{proposition}[Asymptotic behavior of $\mathcal{NW}(n,p)$] For fixed $p$, let $\MAT V_n \sim \mathcal{NW}(n,p)$. Then,
    \begin{equation*}
    \sqrt{n}\ \vvec\left(\MAT V_n-\frac{\MAT I_p}{p}\right) \overset{\textup{(d)}}{\underset{n\to+\infty}{\longrightarrow}}\mathcal{N}_{p^2}\left(0, \frac{1}{p^2}(\MAT I_{p^2}+\MAT K_{pp}) -\frac{2}{p^3} \vvec(\MAT I_p)\vvec(\MAT I_p)^\top\right).
\end{equation*}
\label{prop:asymp_normalized_wishart}
\end{proposition}
\begin{proof}
    The proof uses the asymptotic behavior of the Wishart matrix along with the Delta method. Details are given in Appendix~\ref{app:Asymptotic behavior of Elliptical Wishart distributions}.
\end{proof}

Subsequently, for $\MAT S_n \sim \mathcal{EW}(n,\MAT\Sigma,h_{np})$, its asymptotic distribution depends on the asymptotic behavior of its second-order modular $\mathcal{Q}_n$. Indeed, the stochastic representation enables us to write:
\begin{equation}
    \vvec\left( \frac{p\MAT S_n}{ m_1^{(h_{np})}}-  \MAT \Sigma \right) 
    =p(\MAT \Sigma^{\frac{1}{2}}\otimes \MAT \Sigma^{\frac{1}{2}}) \\ \left[\left( \frac{\mathcal{Q}_n}{ m_1^{(h_{np})}}-1\right)\vvec(\MAT V_n)  +  \vvec\left(\MAT V_n -\frac{\MAT I_p}{p}\right)\right],
    \label{eq:asymptotic_storep}
\end{equation}
which leads to the following proposition:
\begin{proposition}[Asymptotic behavior of $\mathcal{EW}(n,\MAT\Sigma,h_{np})$]
\label{prop:asymp_ew}
    Let $p$ be fixed, $\MAT S_n\sim\mathcal{EW}(n,\MAT\Sigma,h_{np})$ with modular $\mathcal{Q}_n$ and $\mu_n:=m_1^{(h_{np})}$. Assume $\sqrt n\,(\mathcal{Q}_n/\mu_n-1)\overset{\textup{(d)}}{\underset{n\to+\infty}{\longrightarrow}} \mathcal N(0,\sigma^2)$ for some $\sigma^2\ge0$. Then
\begin{equation}
\sqrt n\,\vvec\!\Big(\frac{p\,\MAT S_n}{\mu_n}-\MAT\Sigma\Big)\overset{\textup{(d)}}{\underset{n\to+\infty}{\longrightarrow}}\mathcal N_{p^2}\Big(\MAT 0,\ (\MAT I_{p^2}+\MAT K_{pp})(\MAT\Sigma\otimes\MAT\Sigma)+\Big(\sigma^2-\frac2p\Big)\vvec(\MAT\Sigma)\vvec(\MAT\Sigma)^\top\Big).
\end{equation}
Moreover, $p\MAT S_n/\mu_n\overset{\textup{(P)}}{\underset{n\to+\infty}{\longrightarrow}}\MAT\Sigma$ if and only if $\mathcal{Q}_n/\mu_n\overset{\textup{(P)}}{\underset{n\to+\infty}{\longrightarrow}}1$.
\end{proposition}
\begin{proof}
    The proof leverages Equation \eqref{eq:asymptotic_storep} and Proposition \ref{prop:asymp_normalized_wishart}. Details are provided in Appendix~\ref{app:Asymptotic behavior of Elliptical Wishart distributions}.
\end{proof}

Consequently, we provide in Corollary \ref{cor:exp} some examples of Elliptical Wishart distributions with different behaviors when $n\to +\infty$; its proof is detailed in Appendix \ref{app:Asymptotic behavior of Elliptical Wishart distributions}.

\begin{corollary}\label{cor:exp}
    \begin{enumerate}[label=(\alph*)]
\item  \emph{Wishart $\mathcal{W}(n,\MAT\Sigma)$:} $\mathcal{Q}_n\sim\chi^2_{np}$, $\sigma^2=\frac2p$, which gives \eqref{eq:clt_}.
\item \emph{Kotz Wishart $\mathcal{KW}(n,\MAT\Sigma,\alpha,\beta,R)$:} 
\begin{equation}\label{eq:asymp_kw}\forall \beta>0,\quad \sqrt n\,\vvec\!\Big(\frac{p\,\MAT S_n}{\mu_n}-\MAT\Sigma\Big)\overset{\textup{(d)}}{\underset{n\to+\infty}{\longrightarrow}}\mathcal N_{p^2}\Big(\MAT 0,\ (\MAT I_{p^2}+\MAT K_{pp})(\MAT\Sigma\otimes\MAT\Sigma)+\frac{2}{p}\Big(\frac1\beta -1\Big)\vvec(\MAT\Sigma)\vvec(\MAT\Sigma)^\top\Big).
\end{equation}Consequently, $n^{\frac{3}{2}-\frac{1}{\beta}}\vvec\left( \frac{\MAT S_n}{n}- \frac{\mu_n}{np} \MAT \Sigma \right) 
         \overset{\textup{(d)}}{\underset{n\to+\infty}{\longrightarrow}} \mathcal{N}_{p^2}(\VEC 0,\frac{1}{p^2}\left(\frac{p}{2R\beta}\right)^{\frac{2}{\beta}} [(\MAT I_{p^2}+\MAT K_{pp})(\MAT \Sigma\otimes \MAT\Sigma) +\frac{2}{p}(\frac{1}{\beta}-1)\vvec(\MAT \Sigma)\vvec(\MAT \Sigma)^\top]$.\\
%
\item \emph{$t$- Wishart $t$-$\mathcal{W}(n,\MAT\Sigma,\nu)$:} $\frac{\MAT S_n}{n}\overset{\textup{(d)}}{\underset{n\to+\infty}{\longrightarrow}} \frac{\nu}{\chi^2_\nu}\MAT\Sigma$. There is no concentration and no CLT around $\MAT\Sigma$.
\end{enumerate}
\end{corollary}

%
 We point out that the asymptotic behavior of the Kotz Wishart depends merely on its shape parameter $\beta$. We showed that $\frac{p\MAT S_n}{\mu_n}$ is asymptotically normal at rate $\sqrt{n}$ for every
$\beta>0$; only the coefficient of $\vvec(\MAT\Sigma) \vvec(\MAT\Sigma)^\top$ in the limiting covariance depends on $\beta$. Since $\mu_n \propto n^\frac{1}{\beta}$, the fluctuations of $\frac{\MAT S_n}{n}$ around $\frac{\mu_n}{np}\MAT\Sigma$ are of order $n^{\frac{1}{\beta}-\frac{3}{2}}$.
 Furthermore, we notice the significant difference in the asymptotic behaviors between the $t$-Wishart-distributed matrix and the SCM of IID $t$- vectors. With the notations of Section \ref{sec:compare_scm}, $\sqrt{n} \vvec\left( \frac{\MAT R_n}{n}- \frac{\mu_n}{np} \MAT \Sigma \right) $ (where $\mu_n = np \frac{\nu}{\nu-2}$) is asymptotically normal, thanks to CLT, while $t$- Wishart matrix is not.

To conclude, Elliptical Wishart distributions feature different asymptotic behaviors when the degree of freedom $n$ tends to $+\infty$. Because of the dependence structure of this family, we demonstrated that even if $n$ is large, the probability distribution of $\MAT S_n/n$ does not necessarily approximate a normal distribution like the Wishart case. 

\begin{remark}
The dependence structure of Elliptical Wishart distributions makes them an excellent candidate in the statistical modeling of several signal processing problems with high noise and mislabeling, for \emph{e.g.}, the question of robustly estimating covariance matrices. Note that the state-of-art approach consists of using the so-called $M$-estimators \cite{maronna1976robust}. The underlying idea beyond such estimators is the robustification of the distribution of the signal matrices' columns.    
   The asymptotic distribution $M$-estimators when the number of observations $n \to +\infty$  is more accurately described by a Wishart distribution than by the classical asymptotic normal approximation~\cite{dravskovic2018new}. Meanwhile, the asymptotic behavior of the Elliptical Wishart distribution deviates from the one of the Wishart distribution, as described above. This highlights the potential of Elliptical Wishart distributions to offer more flexibility than $M$-estimators.
 \end{remark}

\section{Numerical experiments}
\label{sec:application}

In this section, we highlight the merit of Elliptical Wishart distributions on real data compared to the Wishart distribution. Indeed, we assess to which extent some encephalographic signals can be fitted with Elliptical Wishart models. In order to perform such a distribution fitting efficiently, we need first to generate samples from Elliptical Wishart distributions in a numerically practical way. This is achieved in Section~\ref{subsec:generation} wherein we also suggest a deft method to draw samples from Inverse Elliptical Wishart distributions. This accounts for a contribution of the paper as our new technique of generating samples drawn from Elliptical and Inverse Elliptical Wishart distributions outperforms the existing technique. Thereafter, we exploit this method in Section~\ref{subsec:fitting} to fit real EEG data in order to show that the Elliptical Wishart distribution offers a more relevant fitting than the Wishart distribution, commonly used as the state-of-the-art model based on the Gaussianity hypothesis.

\subsection{Samples generation}
\label{subsec:generation}
This first subsection is dedicated to developing an efficient way to draw samples from Elliptical and Inverse Elliptical Wishart distributions.
In fact, by definition of the Inverse Elliptical Wishart distribution (Definition~\ref{def:defEIW}), it is enough to be able to generate samples from the Elliptical Wishart distribution. 
Again, to achieve this, we heavily rely on the stochastic representation from Theorem~\ref{thm:stoc rep} and existing results for the Wishart distribution.

A naive way to draw samples from the Elliptical Wishart distribution $\mathcal{EW}(n,\MAT\Sigma, h_{np})$ with center $\MAT\Sigma\in\mathcal{S}^{++}_p$ is to exploit Definition~\ref{def:defEW}.
Indeed, one can simply draw $\VEC{x}\in\mathbb{R}^{np}$ from the multivariate elliptical distribution $\mathcal{E}_{np}(\VEC 0, \MAT I_{np},h_{np})$%
to construct the random matrix $\MAT{X}\sim\mathcal{E}_{p,n}(\MAT 0,\MAT I_p,\MAT I_n,h_{np})$ such that $\vvec(\MAT X)=\VEC x$ (Definition~\ref{def:matrix_elliptical_dist}).
Then, from the Cholesky decomposition $\MAT\Sigma=\MAT{L}\MAT{L}^\top$%
\footnote{
    $\MAT{L}\in\mathbb{R}^{p\times p}$, lower triangular matrix with strictly positive diagonal elements.
},
a random matrix $\MAT{S}\sim\mathcal{EW}(n,\MAT\Sigma, h_{np})$ is obtained by taking $\MAT{S}=\MAT{L}\MAT{X}\MAT{X}^\top\MAT{L}^\top$.
However, this procedure quickly becomes computationally very demanding when increasing the degree of freedom $n$ since it implies generating an elliptical vector of length $np$.
Since the latter usually corresponds to the number of available samples, it is commonly quite big.

To deal with this issue in the Wishart case, the so-called Bartlett decomposition~\cite{bartlett1934xx} is exploited.
It indeed provides a very computationally efficient way to generate random Wishart matrices.
First, to build $\MAT S\sim\mathcal{W}(n,\MAT\Sigma)$, $\MAT R\sim\mathcal{W}(n,\MAT I_p)$ is constructed.
$\MAT S$ is then obtained through $\MAT S = \MAT L\MAT R\MAT L^\top$, where $\MAT\Sigma=\MAT L\MAT L^\top$ is the Cholesky decomposition of $\MAT\Sigma$.
To synthesize $\MAT R$, its Cholesky decomposition $\MAT{R}=\MAT{T}\MAT{T}^\top$ is also exploited.
Indeed, from the Bartlett decomposition, random values $(\MAT T_{k\ell})_{1\leq \ell\leq k \leq p}$ are independent.
Moreover, every diagonal element $\MAT T_{kk}$ ($1\leq k\leq p$) is drawn from the Chi distribution $\chi_{n-k+1}$ with $n-k+1$ degrees of freedom, while the lower triangular elements $\MAT T_{k\ell}$ ($1\leq \ell < k\leq p$) are drawn from the normal distribution $\mathcal{N}(0,1)$.
This yields a simple way to generate $\MAT T$ and thus $\MAT S\sim\mathcal{W}(n,\MAT\Sigma)$.

Similarly to the multivariate case, the stochastic representation can be exploited to simulate random Elliptical Wishart matrices.
To get $\MAT S\sim\mathcal{EW}(n,\MAT\Sigma,h_{np})$, we exploit Theorem~\ref{thm:stoc rep} and set $\MAT S=\mathcal{Q}\MAT L\MAT V\MAT L^\top$, where $\MAT L$ corresponds to the Cholesky decomposition of $\MAT\Sigma$, $\mathcal{Q}$ is the so-called second-order modular random variable and $\MAT V$ follows the Normalized Wishart distribution $\mathcal{NW}(n,p)$.
To construct a proper $\MAT V$, we synthesize $\MAT R\sim\mathcal{W}(n,\MAT I_p)$ with the Bartlett decomposition approach described above and then set $\MAT V = \frac{\MAT R}{\tr(\MAT R)}$.
The proposed method to generate data from the Elliptical Wishart distribution is summarized in Algorithm~\ref{alg:sampling}.
To conclude, the naive method requires $np$ random variates and $O(np^2)$ operations per matrix; Algorithm~\ref{alg:sampling} requires one draw of $\mathcal{Q}$, $p(p+1)/2$ further variates and $O(p^3)$ operations, independently
of $n$.

 \begin{algorithm}
 \footnotesize
\caption{\footnotesize Sample generation from the Elliptical Wishart distribution}
\label{alg:sampling}
\begin{algorithmic}
\State \textbf{Input:} dimension $p$, degrees of freedom $n \geq p$, center $\MAT \Sigma \in \mathcal{S}_p^{++}$, density generator $h_{np}$
\State \textbf{Output:} random matrix $\MAT S\sim \mathcal{EW}(n,\MAT \Sigma,h_{np})$
\vspace*{5pt}
\State Generate $\mathcal{Q}$ with PDF $t\mapsto \frac{\pi^{np/2}}{\Gamma(np/2)} h_{np}(t)t^{\frac{np}{2}-1}$
\State Generate lower triangular $\MAT T\in\mathbb{R}^{p\times p}$ such that:
\begin{itemize}
    \item $\MAT T_{kk}\stackrel{\text{IID}}{\sim} \chi_{n-k+1}$, $1\leq k\leq p$ and,
    \item $\MAT T_{k\ell}\stackrel{\text{IID}}{\sim} \mathcal{N}(0,1)$, $1\leq \ell < k \leq p$, independent from the $\MAT T_{kk}$'s.
\end{itemize}
\State Set $\MAT R=\MAT T\MAT T^\top$
 and  $\MAT V=\frac{\MAT R}{\tr(\MAT R)}$
\State Compute Cholesky decomposition $\MAT\Sigma=\MAT L\MAT L^\top$
\State Let $\MAT S=\mathcal{Q}\MAT L\MAT V\MAT L^\top$
\end{algorithmic}
\end{algorithm}

\subsection{Electroencephalographic data fitting}
\label{subsec:fitting}

We perform some distribution fitting over EEG data to assess the interest of Elliptical Wishart distributions in practice.
We consider the Steady-State Visually Evoked Potential (SSVEP) dataset~\cite{kalunga2016online}, which contains recordings from $12$ subjects. In the SSVEP paradigm, subjects are exposed to a light blinking at various frequencies. Data are thus divided into several classes corresponding to these various stimulation frequencies. In our case, the dataset contains four classes corresponding to stimuli at $13$Hz, $17$Hz, and $21$Hz, along with a resting state. To keep our experiment simple, we limit ourselves to subject $12$, which features the cleanest data. These data contain $160$ trials equally partitioned between the four classes. For each trial, we have $p=8$ electrodes and $n=768$ samples ($3$ seconds sampled at $256$Hz). For further details, the reader is referred to~\cite{kalunga2016online} and to the MOABB platform~\cite{Aristimunha}.

One distribution fitting is achieved for each class. We point out that for the considered paradigm, the raw trials $\MAT X_i$'s are centered by nature. In our experiments, we re-centered them anyway, and then, we rescaled them up to a constant factor ($\times 10^{-6}$) to avoid computational problems. For class $z$, we first estimate the center $\MAT\Sigma^{(z)}$ of the tested distribution from the covariance matrices $\MAT S_i^{(z)}=\MAT X_i^{(z)} \MAT X_i^{(z)\,\top}$ of the trials $\MAT X_i^{(z)}\in\mathbb{R}^{p\times n}$ belonging to class $z$. In this work, we consider the following distributions for the testing: the Wishart distribution $\mathcal{W}(n,\MAT\Sigma^{(z)})$ and the $t$-Wishart distribution $t\textup{-}\mathcal{W}(n,\MAT\Sigma^{(z)},\nu^{(z)})$ with $\nu^{(z)}$ degrees of freedom for the $t$- density generator. 
$\MAT\Sigma^{(z)}$ is set as the maximum likelihood estimator of the distribution.
For the Wishart distribution, it is simply given by $\MAT\Sigma_{\mathcal{W}}^{(z)}=\frac1{nK_z}\sum_i\MAT S_i^{(z)}$, where $K_z$ corresponds to the number of trials belonging to class $z$.
For the $t$-Wishart case, no closed form exists, and we use the iterative algorithm derived in~\cite{ayadi2025elliptical} to estimate it. The resulting estimator is denoted by $\MAT\Sigma_{t\text{-}\mathcal{W}}^{(z)}$ and required to choose $\nu^{(z)}$ before. That is why we estimated it simply using the kurtosis estimator \cite{ayadi2023elliptical,ollila2019optimal}:
$$\nu^{(z)}=4+\frac{2}{\max(\kappa^{(z)},0)}; \qquad \kappa^{(z)}=\frac{1}{np(np+2)K_z}\sum_{k=1}^{K_z}\tr\left(\MAT\Sigma_{\mathcal{W}}^{(z)-1}\MAT S_k^{(z)}\right)^2-
1.$$
The actual values used in our experiments are given in Table~\ref{tab:optimal_dof}. 
\begin{table}[H]
\centering
\begin{tabular}{|c|c|c|c|}
\hline
$\nu^{(13)}$ & $\nu^{(17)}$ & $\nu^{(21)}$ & $\nu^{(\text{resting})}$ \\ \hline
$56.35$ & $63.58$ & $54.02$ & $45.57$ \\ \hline
\end{tabular}
\caption{Degrees of freedom to generate $t$-Wishart samples for fitting EEG data.}
\label{tab:optimal_dof}
\end{table}
The actual fitting procedure consists in performing goodness-of-fit Kolmogorov-Smirnov statistical tests of the traces, traces of the third power, $\ell_2$-norms, $\ell_2$-norms of the square, $\ell_2$-norms of the third power, and log-determinants of real EEG covariances $\{\MAT S_i^{(z)}\}_{i=1}^{K_z}$, for every class $z\in\{13,17,21,\text{resting}\}$.
Such tests aim to assess to which extent these samples are drawn from the distributions of the traces, traces of the third power, $\ell_2$-norms, $\ell_2$-norms of the square, $\ell_2$-norms of the third power, and log-determinants, respectively, of the Wishart distribution $\mathcal{W}(n,\MAT\Sigma_{\mathcal{W}}^{(z)})$ and $t$-Wishart distribution $t\textup{-}\mathcal{W}(n,\MAT\Sigma_{t\text{-}\mathcal{W}}^{(z)},\nu^{(z)})$.
Performing a one-sample Kolmogorov-Smirnov test requires computing the exact CDFs of these quantities. To avoid the computational burden of such calculations, we instead generate numerous random samples from these distributions using the procedure introduced in Section~\ref{subsec:generation}. We then empirically estimate the CDFs, and consequently, we rely on two-sample tests instead of one-sample tests.
We recall that the null hypothesis is that the compared samples are drawn from the same distribution.
The $p$-values of the achieved Kolmogorov-Smirnov tests 
are provided in Table~\ref{tab:fitting}. To give a visual inspection for some of the tests' results, we plotted in Figure
\ref{fig:fitting_norm_subj12} the empirical CDFs of the $\ell_2$-norms of the real EEG covariances $\{\MAT S_i^{(z)}\}_{i=1}^{K_z}$, for every class $z\in\{13,17,21,\text{resting}\}$ along with the empirical CDFs of the $\ell_2$-norms of simulated samples drawn from the Wishart and the $t$-Wishart distributions. 

From the obtained table and figure, we conclude that the $t$-Wishart distribution better fits the covariances than the Wishart distribution for every class.
While the correspondence between the Wishart distribution and the data at hand is quite poor, the $t$-Wishart distribution is a pretty good fit.
Even though the considered dataset is very specific and quite small, this clearly demonstrates the practical interest of such a distribution. For instance, this spurs the design of a Bayesian classifier of EEG signals that exploits $t$-Wishart in a discriminant analysis framework, as presented in \cite{ayadi2023t}, wherein the proposed classifier outperforms the state-of-the-art method on the tested datasets.

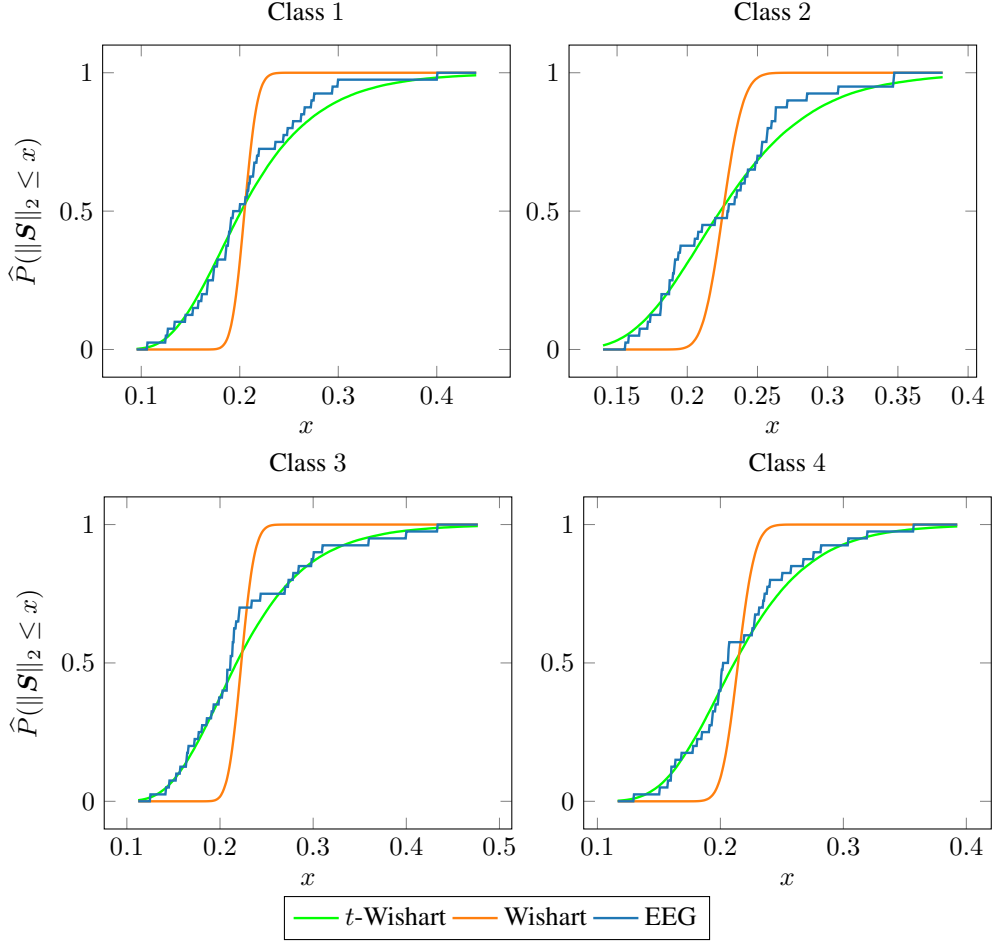
\begin{figure}[t!]
 \footnotesize
\centering
\begin{tikzpicture}
    \begin{axis}[
        width=0.52\linewidth,
        height=4.2cm,
        at     ={(0,0)},
        xlabel      = {\footnotesize $x$},
        ylabel = {\footnotesize $\widehat{P}(\|\MAT S\|_2\leq x)$},
        title = {\footnotesize $13$ Hz},
        legend style={at={(0,0)},    
                    anchor=north,legend columns=-1},
        legend to name={mylegend1},
        ]
        \addplot[color=green,line width=0.9pt] table [x=x,y=t_wishart_cdf,col sep=comma] {./figures_/cdf_norm_class_13.txt};
        \addlegendentry{$t$-Wishart};
        \addplot[color=myorange,line width=0.9pt] table  [x=x,y=wishart_cdf,col sep=comma] {./figures_/cdf_norm_class_13.txt};
        \addlegendentry{Wishart};
        \addplot[color=myblue,line width=0.9pt] table  [x=x,y=eeg_cdf,col sep=comma] {./figures_/cdf_norm_class_13.txt};
        \addlegendentry{EEG};
    \end{axis}
\end{tikzpicture}
\begin{tikzpicture}
    \begin{axis}[
        width=0.52\linewidth,
        height=4.2cm,
        at     ={(0.48\linewidth,0)},
        xlabel      = {\footnotesize $x$},
        title = {\footnotesize $17$ Hz},
        legend style={at={(0,0)},    
                    anchor=north,legend columns=-1},
        legend to name={mylegend2},
        ]
        \addplot[color=green,line width=0.9pt] table [x=x,y=t_wishart_cdf,col sep=comma] {./figures_/cdf_norm_class_17.txt};
        \addlegendentry{$t$-Wishart};
        \addplot[color=myorange,line width=0.9pt] table  [x=x,y=wishart_cdf,col sep=comma] {./figures_/cdf_norm_class_17.txt};
        \addlegendentry{Wishart};
        \addplot[color=myblue,line width=0.9pt] table  [x=x,y=eeg_cdf,col sep=comma] {./figures_/cdf_norm_class_17.txt};
        \addlegendentry{EEG};
    \end{axis}
\end{tikzpicture}
\begin{tikzpicture}
    \begin{axis}[
        width=0.52\linewidth,
        height=4.2cm,
        at     ={(0,-4.2cm)},
        xlabel      = {\footnotesize $x$},
        ylabel = {\footnotesize $\widehat{P}(\|\MAT S\|_2\leq x)$},
        title = {\footnotesize $21$ Hz},
        legend style={at={(0,0)},    
                    anchor=north,legend columns=-1},
        legend to name={mylegend3},
        ]
        \addplot[color=green,line width=0.9pt] table [x=x,y=t_wishart_cdf,col sep=comma] {./figures_/cdf_norm_class_21.txt};
        \addlegendentry{$t$-Wishart};
        \addplot[color=myorange,line width=0.9pt] table  [x=x,y=wishart_cdf,col sep=comma] {./figures_/cdf_norm_class_21.txt};
        \addlegendentry{Wishart};
        \addplot[color=myblue,line width=0.9pt] table  [x=x,y=eeg_cdf,col sep=comma] {./figures_/cdf_norm_class_21.txt};
        \addlegendentry{EEG};
    \end{axis}
\end{tikzpicture}
\begin{tikzpicture}
    \begin{axis}[
        width=0.52\linewidth,
        height=4.2cm,
        at     ={(0.48\linewidth,-4.2cm)},
        xlabel      = {\footnotesize $x$},
        title = {\footnotesize Resting state},
        legend style={at={(0,0)},    
                    anchor=north,legend columns=-1},
        legend to name={mylegend4},
        ]
        \addplot[color=green,line width=0.9pt] table [x=x,y=t_wishart_cdf,col sep=comma] {./figures_/cdf_norm_class_rest.txt};
        \addlegendentry{$t$-Wishart};
        \addplot[color=myorange,line width=0.9pt] table  [x=x,y=wishart_cdf,col sep=comma] {./figures_/cdf_norm_class_rest.txt};
        \addlegendentry{Wishart};
        \addplot[color=myblue,line width=0.9pt] table  [x=x,y=eeg_cdf,col sep=comma] {./figures_/cdf_norm_class_rest.txt};
        \addlegendentry{EEG};
    \end{axis}
\end{tikzpicture}
     \caption{Empirical CDF of the $\ell_2$-norms of the covariance matrices extracted from EEG signals along with the CDF of the $\ell_2$-norm of Wishart matrix and the CDF of the $\ell_2$-norm of $t$-Wishart matrix, obtained as the empirical CDFs of a large number ($10^5$) of $\ell_2$-norms of generated samples from Wishart $\mathcal{W}(n,\MAT\Sigma_{\mathcal{W}}^{(z)})$ and $t$-Wishart $t$-$\mathcal{W}(n,\MAT\Sigma_{t\text{-}\mathcal{W}}^{(z)},\nu^{(z)})$ respectively.  }
         \label{fig:fitting_norm_subj12}
\end{figure}

\begin{table}[h!]
\footnotesize
    \centering
    \renewcommand{\arraystretch}{1.1}
    \begin{tabular}{clcccc}
        \hline
        \begin{tabular}[c]{@{}c@{}} Fitted\\ quantity\end{tabular} &  Distribution &  $13$ Hz &  $17$Hz &  $21$Hz &  Resting 
        \\
        \hline
        \multirow{2}{*}{ $\tr(\MAT S)$}
        & Wishart & $10^{-8}$  &$10^{-5}$  &$10^{-5}$ &$10^{-7}$
        \\
        & $t$-Wishart  &$0.3166$ 
        & $0.6549$ 
        & $0.5946$ 
        & $0.7246$
        \\
        \hline
        \multirow{2}{*}{$\tr(\MAT S^3)$}
        & Wishart &$10^{-8}$  &$10^{-4}$  &$10^{-5}$ &$10^{-6}$
        \\
        & $t$-Wishart 
        & $0.3566$ 
        & $0.9619$ 
        & $0.3328$ 
        & $0.5452$
        \\
        \hline
        \multirow{2}{*}{ $\|\MAT S\|_2$}
        & Wishart &$10^{-7}$  &$10^{-5}$  &$10^{-5}$ &$10^{-6}$
        \\
        & $t$-Wishart  &$0.3526$ & $0.8697$ & $0.708$ & $0.4327$
        \\
        \hline
        \multirow{2}{*}{ $\|\MAT S^2\|_2$}
        & Wishart &$10^{-7}$  &$10^{-4}$  &$10^{-4}$ &$10^{-6}$
        \\
        & $t$-Wishart  &$0.5125$ 
        &$0.7404$ 
        &$0.6763$ 
        &$0.3733$
        \\
        \hline
        \multirow{2}{*}{ $\|\MAT S^3\|_2$}
        & Wishart &$10^{-6}$  & $10^{-4}$ &$10^{-4}$ &$10^{-6}$
        \\
        & $t$-Wishart &$0.6$ 
        &$0.9106$ & $0.3553$ & $0.3572$
        \\
        \hline
        \multirow{2}{*}{ $-\log_{10}|\MAT S|$}
        & Wishart &$10^{-15}$  &$10^{-11}$  &$10^{-10}$ &$10^{-9}$
        \\
        & $t$-Wishart  & $0.128$ & $0.1014$ & $0.1516$ & $0.2971$
        \\
        \hline
    \end{tabular}
    \caption{p-values of the two-sample Kolmogorov-Smirnov tests}
    \label{tab:fitting}
\end{table}

\section{Conclusion and Perspectives}
\label{sec:conclusion}

This paper studies some extensions of the Wishart and Inverse Wishart distributions: the Elliptical Wishart and Inverse Elliptical Wishart distributions, which allow us to better take into account noise and outliers at the covariance level.
Our main contribution is to obtain their stochastic representations.  
From there, some statistical properties of the Elliptical Wishart and Inverse Elliptical Wishart that are not covered in the literature are obtained.
Moreover, we investigated the dependence structure of the Elliptical Wishart distributions and studied their asymptotic behavior.
Finally, fitting experiments on real EEG data show their practical interest.

In the future, it would be interesting to test the classification-clustering methods based on Elliptical Wishart distributions on various signal datasets. Moreover, it is worth investigating the merit of these distributions for geodesic regression or principal geodesic analysis (PGA) on the manifold of SPDs for dimensionality reduction.

\section*{Acknowledgments}
This work is partially supported by a public grant overseen by the French National Research Agency (ANR) through the program UDOPIA, a project funded by the ANR-20-THIA-0013-01.
This research was also supported by DATAIA Convergence Institute as part of the ``Programme d’Investissement d’Avenir", (ANR-17-CONV-0003) operated by laboratoire des Signaux et Systèmes.

\bibliographystyle{plain}
\bibliography{ref.bib}

\newpage
\appendix

\section{Missing statistical properties of Elliptical Wishart and Inverse Elliptical Wishart}
\subsection{Normalized Wishart distribution}
\label{app:Normalized Wishart distribution}
\begin{proof}[\textbf{Proof of Proposition~\ref{prop:uniform}}]
    Let $\MAT S \sim \mathcal{W}(n,\MAT I_p)$ such that $\MAT V = \MAT S / \tr(\MAT S)$.
    From Theorem~\ref{thm:stoc rep}, $\MAT S = \mathcal{Q}\MAT V$, where $\mathcal{Q}\sim\chi^2_{np}$ and $\MAT V \stackrel{(d)}{=}\MAT U\MAT U^\top$ with $\MAT U$ uniformly distributed on the unit sphere $\mathcal{C}_{p,n}$.
    Moreover, $\mathcal{Q}$ and $\MAT U$ are independent (hence, $\mathcal{Q}$ and $\MAT V$ are also independent).
    First, to prove 1., remembering that $E[\mathcal{Q}]=np$ and $E[\MAT S]=n\MAT I_p$ (from~\eqref{eq:expecW}), it is enough to notice that
    \begin{equation*}
        E[\MAT S]=n\MAT I_p=E[\mathcal{Q}]E[\MAT V] = np E(\MAT V).
    \end{equation*}
    To prove 2., recall that $E[\mathcal{Q}^2]=np(np+2)$ and from~\eqref{eq:varW}, $E[\vvec(\MAT S)\vvec(\MAT S)^\top]=n(\MAT I_{p^2}+\MAT K_{p,p})+n^2\vvec(\MAT I_p)\vvec(\MAT I_p)^\top$.
    It is then enough to compute
    \begin{equation*}
        E[\vvec(\MAT S)\vvec(\MAT S)^\top]
        = n(\MAT I_{p^2}+\MAT K_{p,p})+n^2\vvec(\MAT I_p)\vvec(\MAT I_p)^\top\\ 
        =E[\mathcal{Q}^2]E[\vvec(\MAT V)\vvec(\MAT V)^\top].
    \end{equation*}
    To prove 3., it is enough to notice that, by construction, $\MAT S$ is invertible with probability $1$.
    To prove 4., remember that $E[\mathcal{Q}^{-1}]=\frac{1}{np-2}$ and $E[\MAT S^{-1}]=\frac{1}{n-p-1}\MAT I_p$ (Equation~\eqref{eq:expecIW}).
    To conclude, it is then enough to observe
    \begin{equation}
        E[\MAT S^{-1}] = \frac{1}{n-p-1}\MAT I_p =
        E[\mathcal{Q}^{-1}]E[\MAT V^{-1}] = \frac{1}{np-2} E[\MAT V^{-1}].    \end{equation}
    Finally, to prove 5., recall $E[\mathcal{Q}^{-2}]=\frac{1}{(np-2)(np-4)}$ and, from~\eqref{eq:varIW},
    \begin{equation*}
        E[\vvec(\MAT S^{-1})\vvec(\MAT S^{-1})^\top]=\frac{1}{(n-p)(n-p-1)(n-p-3)}
        [(\MAT I_{p^2}+\MAT K_{p,p})+(n-p-2)\vvec(\MAT I_p)\vvec(\MAT I_p)^\top].
    \end{equation*}
    The result follows immediately from combining these with
    \begin{equation*}
        E[\vvec(\MAT S^{-1})\vvec(\MAT S^{-1})^\top] = E[\mathcal{Q}^{-2}]E[\vvec(\MAT V^{-1})\vvec(\MAT V^{-1})^\top].
    \end{equation*}
    This concludes the proof.
\end{proof}

\begin{proof}[\textbf{Proof of Lemma~\ref{lem:zonal unif}}]
    Let $\MAT S \sim \mathcal{W}(n,\MAT I_p)$ such that $\MAT V = \MAT S / \tr(\MAT S)$.
    From Theorem~\ref{thm:stoc rep}, $\MAT S = \mathcal{Q}\MAT V$, where $\mathcal{Q}\sim\chi^2_{np}$ and $\MAT V \stackrel{(d)}{=}\MAT U\MAT U^\top$ with $\MAT U$ uniformly distributed on the unit sphere $\mathcal{C}_{p,n}$.
    Furthermore, $\mathcal{Q}$ and $\MAT V$ are independent.
    Therefore, since $C_\kappa(\alpha \MAT A)= \alpha^{k}C_\kappa(\MAT A)$ for $\alpha \in\mathbb{R}$ and $\MAT A \in \mathcal{S}_p$,
    \begin{equation*}
        E[C_\kappa(\MAT Z\MAT S)] =
        E[C_\kappa(\mathcal{Q} \MAT Z\MAT V)] 
        =
        E[\mathcal{Q}^k C_\kappa( \MAT Z\MAT V)]
        =
        E[\mathcal{Q}^k] E[ C_\kappa( \MAT Z\MAT V)]. 
    \end{equation*}
    From~\cite{simon2002probability}, $E[\mathcal{Q}^k]=2^k\frac{\Gamma\left(\frac{np}{2}+k\right)}{\Gamma\left(\frac{np}{2}\right)}$.
    The result stems from \eqref{eq:zonalW}.
\end{proof}

\subsection{Elliptical Wishart distribution}
\label{app:Elliptical Wishart distribution}
\begin{proof}[\textbf{Proof of Proposition~\ref{prop:variance}}]
    From Theorem~\ref{thm:stoc rep}, $\MAT S=\mathcal{Q}\MAT\Sigma^{1/2}\MAT U\MAT U^\top\MAT\Sigma^{1/2}$.
    Therefore, since $(\MAT C^\top\otimes\MAT A)\vvec(\MAT B)=\vvec(\MAT A\MAT B\MAT C)$, we have $\vvec(\MAT S)=\mathcal{Q}(\MAT\Sigma^{1/2}\otimes \MAT\Sigma^{1/2})\vvec(\MAT U\MAT U^\top)$.
    Since $\MAT U$ and $\mathcal{Q}$ are independent,
    \begin{equation*}
        E[\vvec(\MAT S)\vvec(\MAT S)^\top] = E[\mathcal{Q}^2](\MAT\Sigma^{1/2}\otimes \MAT\Sigma^{1/2})
        E[\vvec(\MAT U\MAT U^\top)\vvec(\MAT U\MAT U^\top)^\top](\MAT\Sigma^{1/2}\otimes \MAT\Sigma^{1/2}).
    \end{equation*}
    The result is obtained by using 2. from Proposition~\ref{prop:uniform}.
\end{proof}

\begin{proof}[\textbf{Proof of (ii) in Proposition~\ref{prop:char new}}]
For $\MAT{Z}\in\mathcal{S}_p$, we introduce the notation
\begin{equation}\label{eq:cfseries}
T_k(\MAT{Z}):=\frac{i^k}{k!}\,m_k^{(h_{np})}\,\frac{\Gamma\!\left(\frac{np}{2}\right)}{\Gamma\!\left(\frac{np}{2}+k\right)}\sum_{|\kappa|=k}\left(\frac n2\right)_\kappa C_\kappa(\MAT{Z}\MAT{\Sigma}),\qquad k\in\mathbb N,
\end{equation}
so that $\sum_{k\ge0}T_k(\MAT{Z})$ is the series~\eqref{eq:charEW}. We prove here that when $\rho(\MAT{Z}\MAT{\Sigma})>t_0$,
$$
\sup_{k\in\mathbb N}|T_k(\MAT{Z})|=+\infty ;
$$
in particular, the series~\eqref{eq:charEW} diverges at $\MAT{Z}$.\\
Let $\MAT{A}:=\MAT{\Sigma}^{1/2}\MAT{Z}\MAT{\Sigma}^{1/2}\in\mathcal{S}_p$, with spectral decomposition $\MAT{A}=\MAT{H}\MAT{\Lambda}\MAT{H}^\top$, where $\MAT{H}$ is orthogonal and $\MAT{\Lambda}=\mathrm{diag}(\lambda_1,\dots,\lambda_p)$. Since $\MAT{A}=\MAT{\Sigma}^{-1/2}(\MAT{\Sigma}\MAT{Z})\MAT{\Sigma}^{1/2}$ is similar to $\MAT{\Sigma}\MAT{Z}$, and $\MAT{\Sigma}\MAT{Z}$ and $\MAT{Z}\MAT{\Sigma}$ have the same eigenvalues,
$$
\rho:=\max_j|\lambda_j|=\rho(\MAT{Z}\MAT{\Sigma})
\qquad\text{and}\qquad
C_\kappa(\MAT{A})=C_\kappa(\MAT{Z}\MAT{\Sigma})\quad\text{for every }\kappa .
$$
The second identity holds because zonal polynomials are symmetric functions of the eigenvalues. By assumption, $\rho>t_0$; in particular $t_0<\infty$ and $\rho>0$.\\
Let $\MAT{V}\sim\mathcal{NW}(n,p)$. By Lemma~\ref{lem:zonal unif} and $(\tr(\MAT{M}))^k=\sum_{|\kappa|=k}C_\kappa(\MAT{M})$,
\begin{equation}\label{eq:Tk}
a_k:=E\big[\tr(\MAT{A}\MAT{V})^k\big]=\frac{\Gamma\!\left(\frac{np}{2}\right)}{\Gamma\!\left(\frac{np}{2}+k\right)}\sum_{|\kappa|=k}\left(\frac n2\right)_\kappa C_\kappa(\MAT{Z}\MAT{\Sigma}),
\quad\text{hence}\quad
T_k(\MAT{Z})=\frac{i^k}{k!}\,m_k^{(h_{np})}\,a_k .
\end{equation}
\smallskip\noindent\textbf{\emph{Step 1}}: 
By Remark~\ref{rk:bilodeau}, $\MAT{H}^\top\MAT{V}\MAT{H}\stackrel{(d)}{=}\MAT{V}$, hence $\tr(\MAT{A}\MAT{V})\stackrel{(d)}{=}\sum_{j=1}^p\lambda_j\MAT V_{jj}$. By the Bartlett decomposition (Section~\ref{subsec:generation},~\cite{bartlett1934xx}), $\MAT{V}\stackrel{(d)}{=}\MAT{T}\MAT{T}^\top/\tr(\MAT{T}\MAT{T}^\top)$. Here $\MAT{T}$ is lower triangular with independent entries $\MAT T_{jj}\sim\chi_{n-j+1}$ and $\MAT T_{j\ell}\sim\mathcal N(0,1)$ for $\ell<j$. Since the first row of $\MAT{T}$ is $(\MAT T_{11},0,\dots,0)$,
$$
\MAT V_{11}\stackrel{(d)}{=}\frac{R_1}{R_1+R'},\qquad R_1:=\MAT T_{11}^2\sim\chi^2_n,\qquad R':=\sum_{(j,\ell)\neq(1,1),\ \ell\le j}\MAT T_{j\ell}^2 ,
$$
where $R_1$ and $R'$ are independent. If $p=1$, then $R'=0$.

\textbf{\emph{Step 2}}: Let us compute a lower bound on $|T_k(\MAT{Z})|$ for even $k$. Since $\rho>t_0$, we can fix $\delta\in(0,1)$ such that $x:=\rho(1-\delta)>t_0$. As $\MAT{V}\succeq\MAT{0}$ and $\tr(\MAT{V})=1$, the diagonal entries satisfy $\MAT V_{jj}\ge0$ and $\sum_j\MAT V_{jj}=1$. Therefore
$$
\Big|\sum_{j}\lambda_j\MAT V_{jj}\Big|\ \ge\ \rho \MAT V_{11}-\rho\,(1-\MAT V_{11})=\rho\,(2\MAT V_{11}-1)\ \ge\ x
\quad\text{on the event}\quad
\mathcal E:=\{\MAT  V_{11}\ge1-\delta/2\}.
$$
The event $\mathcal E$ contains $\{R_1\ge1\}\cap\{R'\le\delta/(2-\delta)$. Both events have positive probability: $R_1$ has a positive density on $(0,\infty)$, and $R'$ is either $0$ or a finite sum of independent $\chi^2$ variables. They are independent, so $\mathbb P(\mathcal E)>0$. For even $k$, the variable $\big(\sum_j\lambda_j\MAT V_{jj}\big)^k$ is nonnegative, and it is at least $x^k$ on $\mathcal E$. Hence $a_k\ge \mathbb P(\mathcal E)\,x^k$, and by~\eqref{eq:Tk}
\begin{equation}\label{eq:lower}
|T_k(\MAT{Z})|\ \ge\ \mathbb P(\mathcal E)\,\frac{m_k^{(h_{np})}}{k!}\,x^k\qquad\text{for every even }k.
\end{equation}

\textbf{\emph{Step 3}}:  Assume by contradiction that $|T_k(\MAT{Z})|\le B<\infty$ for all $k$. Let $y\in(t_0,x)$. Since $\mathcal{Q}\ge0$, we have $e^{y\mathcal{Q}}\le2\cosh(y\mathcal{Q})=2\sum_{k\ \mathrm{even}}(y\mathcal{Q})^k/k!$. By Fubini's theorem and~\eqref{eq:lower},
\[
E\big[e^{y\mathcal{Q}}\big]\ \le\ 2\sum_{k\ \mathrm{even}}\frac{m_k^{(h_{np})}}{k!}\,y^k\ \le\ \frac{2B}{\mathbb P(\mathcal E)}\sum_{k\ \mathrm{even}}\Big(\frac{y}{x}\Big)^k<\infty .
\]
By definition of $t_0$, this gives $t_0\ge y$, which contradicts $y>t_0$. Hence $\sup_k|T_k(\MAT{Z})|=+\infty$, the terms of the series in \eqref{eq:charEW} do not tend to $0$, and the series diverges.
    \end{proof}

\begin{proof}[\textbf{Proof of Proposition~\ref{prop:kronW}}]
\label{app:kronWproof}
The Kronecker moment of $\MAT S\sim \mathcal{W}(n,\MAT\Sigma)$ of arbitrary order $k$ exists thanks to the dominated convergence theorem. In fact, for $1\leq i,j\leq p^k$, 
$|(\otimes^k \MAT S)_{ij}| \leq  \|\MAT S\|_{2}^k \leq \mathcal{Q}^k \|\MAT \Sigma\|_{2}^k$ and $E[\mathcal{Q}^k] =2^k\frac{\Gamma(\frac{np}{2}+k)}{\Gamma(\frac{np}{2})}<\infty$.

Let $f $ denote the probability density function of the Wishart distribution $\mathcal{W}(n,\MAT\Sigma)$. 
Using the fact that $\frac{\partial f (\MAT S)}{\partial \MAT\Sigma^{-1}} = \left(\frac{n}{2} \MAT\Sigma -\frac{1}{2}\MAT S\right)f (\MAT S)$, we get:
\begin{multline}
        \frac{\partial E[\otimes^k\MAT S] }{\partial \MAT\Sigma^{-1}} = \int_{\mathcal{S}_p^{++}} \frac{\partial }{\partial \MAT \Sigma^{-1}}\left((\otimes^k \MAT S) \otimes f (\MAT S)\right)d\MAT S
        =\int_{\mathcal{S}_p^{++}} (\otimes^k \MAT S) \otimes \frac{\partial f (\MAT S)}{\partial \MAT \Sigma^{-1}}d\MAT S\\
        =\frac{n}{2}\int_{\mathcal{S}_p^{++}} (\otimes^k \MAT S) \otimes \MAT\Sigma \ f (\MAT S) d\MAT S-\frac{1}{2}\int_{\mathcal{S}_p^{++}} (\otimes^{k+1} \MAT S) f (\MAT S) d\MAT S
        = \frac{n}{2} E(\otimes^k \MAT S)\otimes \MAT\Sigma - \frac{1}{2}E[\otimes^{k+1}\MAT S]
    \label{eq:recursive wishart}
    \end{multline}
    where $d\MAT S = \prod_{i\leq j}d\MAT S_{ij}$ and the first equality results from the differentiation under the integral. This leads to the recursion in Equation \eqref{eq:kronW_rec}.
Since $\otimes^{k+1}\MAT S=\MAT S \otimes \left(\otimes^k \MAT S\right)$, then, by applying the formula of the derivative of a Kronecker product:
\begin{equation}
    \frac{\partial \otimes^{k+1}\MAT S}{\partial \MAT\Sigma^{-1}} = \MAT S \otimes \frac{\partial \otimes^k \MAT S}{\partial \MAT\Sigma^{-1}} + \left(\MAT K_{p,p^k}\otimes \MAT I_p\right)\left(\otimes^k\MAT S \otimes \frac{\partial \MAT S}{\partial \MAT\Sigma^{-1}}\right)
    \left(\MAT K_{p^k,p}\otimes \MAT I_p\right)
    \label{eq:rec}
\end{equation}
Accordingly, with a simple recursion, it can be shown that:
\begin{equation}
    \frac{\partial \otimes^k \MAT S}{\partial\MAT \Sigma^{-1}} = \sum_{l=0}^{k-1}\left(\MAT I_{p^l}\otimes \MAT K_{p,p^{k-1-l}}\otimes \MAT I_{p}\right)\left(\otimes^{k-1}\MAT S \otimes \frac{\partial \MAT S}{\partial \MAT \Sigma^{-1}}\right)
    \left(\MAT I_{p^l}\otimes \MAT K_{p^{k-1-l},p}\otimes \MAT I_{p}\right),
    \label{eqs:rec bis}
\end{equation}
and thus:
\begin{equation}
    \vvec\left(\frac{\partial \otimes^k \MAT S}{\partial\MAT \Sigma^{-1}}\right) = \left[\sum_{l=0}^{k-1}(\MAT H_{(k,l)}\otimes \MAT H_{(k,l)})\right]
    (\MAT I_{p^{k-1}}\otimes \MAT K_{p^2,p^{k-1}}\otimes \MAT I_{p^2})\left\{\vvec\left(\otimes^{k-1}\MAT S\right)
    \otimes \vvec\left( \frac{\partial \MAT S}{\partial \MAT\Sigma^{-1}}\right)\right\}.
    \label{eq:eqvec}
\end{equation}

Now, let us compute $\vvec\left(\frac{\partial \MAT S}{\partial \MAT\Sigma^{-1}}\right)$. For this, we write $\MAT S\stackrel{(d)}{=}\MAT\Sigma^{1/2}\MAT R\MAT\Sigma^{1/2}$ where $\MAT R$ does not depend on $\MAT\Sigma$. Therefore,
\begin{equation}
    \frac{\partial \MAT S}{\partial \MAT \Sigma^{-1}} = \frac{\partial \MAT\Sigma^{1/2}}{\partial\MAT \Sigma^{-1}}\left(\MAT R\MAT\Sigma^{1/2}\otimes \MAT I_p\right)+\left(\MAT \Sigma^{1/2}\otimes \MAT I_p\right) \frac{\partial \MAT R\MAT\Sigma^{1/2}}{\partial \MAT\Sigma^{-1}}
    \label{eq:dev1}
\end{equation}
and,
\begin{equation}
    \frac{\partial \MAT R\MAT\Sigma^{1/2}}{\partial \MAT\Sigma^{-1}} = \left(\MAT R\otimes \MAT I_p\right)\frac{\partial \MAT\Sigma^{1/2}}{\partial \MAT\Sigma^{-1}}.
    \label{eq:dev2bis}
\end{equation}
So,
\begin{multline}
    \vvec\left(\frac{\partial \MAT S}{\partial\MAT \Sigma^{-1}} \right)
    = \vvec\left(\frac{\partial \MAT\Sigma^{1/2}}{\partial \MAT\Sigma^{-1}}\left(\ \MAT R\MAT\Sigma^{1/2} \otimes \MAT I_p\right) +\left( \MAT\Sigma^{1/2}\otimes \MAT I_p\right) \left(\MAT  R\otimes \MAT I_p\right)\frac{\partial \MAT\Sigma^{1/2}}{\partial \MAT\Sigma^{-1}}\right)\\
    = \vvec\left(\frac{\partial \MAT\Sigma^{1/2}}{\partial\MAT \Sigma^{-1}}\left(\MAT R\MAT\Sigma^{1/2}\otimes \MAT I_p\right)+\left(\MAT \Sigma^{1/2}\MAT R\otimes \MAT I_p\right)\frac{\partial \MAT\Sigma^{1/2}}{\partial \MAT \Sigma^{-1}}\right)\\
    =(\MAT I_{p^2}\otimes\MAT \Sigma^{1/2}\MAT R\otimes \MAT I_p+\MAT\Sigma^{1/2}\MAT R\otimes \MAT I_p\otimes \MAT I_{p^2}) 
    \vvec\left(\frac{\partial \MAT\Sigma^{1/2}}{\partial \MAT\Sigma^{-1}}\right).
    \label{eq:7}
\end{multline}
Let us compute $\vvec\left(\frac{\partial \MAT\Sigma^{1/2}}{\partial \MAT\Sigma^{-1}} \right)$ step by step. First, by differentiating the equality $\MAT\Sigma^{1/2} \cdot \MAT\Sigma^{1/2}\cdot \MAT\Sigma^{-1}=\MAT I_p$, one gets $$d\MAT\Sigma^{1/2} \cdot \MAT\Sigma^{-1/2} + \MAT\Sigma^{1/2} \cdot d\MAT\Sigma^{1/2} \cdot \MAT\Sigma^{-1} 
         +\MAT\Sigma\cdot d\MAT\Sigma^{-1} = 0.$$
         The multiplication by $\MAT\Sigma^{-1}$ leads to
         $$\MAT\Sigma^{-1}\cdot d\MAT\Sigma^{1/2}\cdot\MAT\Sigma^{-1/2} + \MAT\Sigma^{-1/2} \cdot d\MAT\Sigma^{1/2}\cdot\MAT \Sigma^{-1}=-d\MAT\Sigma^{-1}.$$ Now, using the vectorization operator, one gets:
         $$\left(\MAT\Sigma^{-\frac{1}2}\otimes\MAT\Sigma^{-1}+\MAT\Sigma^{-1}\otimes \MAT\Sigma^{-\frac{1}2}\right)\vvec(d\MAT\Sigma^{\frac{1}2})=-\vvec(d\MAT\Sigma^{-1}),$$
         where we use $\vvec(\MAT A\MAT B\MAT C)=(\MAT C^\top\otimes \MAT A)\vvec(\MAT B)$. Finally, since $\vvec$ is linear, we have $\vvec(d\cdot)=d\vvec(\cdot)$, we conclude that:
         $$d\vvec(\MAT\Sigma^{\frac{1}2})=-\left[\MAT\Sigma^{-\frac{1}2}\otimes\MAT\Sigma^{-1}+\MAT\Sigma^{-1}\otimes \MAT\Sigma^{-\frac{1}2}\right]^{-1} d\vvec(\MAT\Sigma^{-1}).$$
Therefore,
\begin{align*}
    \frac{\partial \vvec(\MAT\Sigma^{1/2})}{\partial \vvec(\MAT\Sigma^{-1})^\top}&= -\frac{1}{4}(\MAT I_{p^2}+\MAT K_{pp})\left(\MAT\Sigma^{-1/2}\otimes\MAT\Sigma^{-1}+\MAT\Sigma^{-1}\otimes \MAT\Sigma^{-1/2}\right)^{-1}
    (\MAT I_{p^2}+\MAT K_{pp})\\
    &=-\frac{1}{2}(\MAT I_{p^2}+ \MAT K_{pp})(\MAT\Sigma^{-1}\oplus \MAT\Sigma^{-1/2})^{-1},
    \label{eq:dev3}
\end{align*}
with the notation $\MAT A \oplus \MAT B = \MAT A \otimes \MAT B + \MAT B \otimes \MAT A$.
Finally, using the fact that: 
\begin{equation*}
    \vvec\left(\frac{\partial \MAT Y}{\partial \MAT X}\right)= (\MAT K_{pp}\otimes \MAT K_{pp})(\MAT I_p\otimes\MAT K_{pp}\otimes \MAT I_p) \vvec\left(\frac{\partial \vvec(\MAT Y)}{\partial \vvec(\MAT X)^\top}\right),
\end{equation*}
we get:
\begin{equation}
    \vvec\left(\frac{\partial \MAT\Sigma^{1/2}}{\partial \MAT\Sigma^{-1}}\right) = -(\MAT K_{pp}\otimes \MAT K_{pp})(\MAT I_p\otimes \MAT K_{pp}\otimes \MAT I_p)
    \vvec\left(\frac{1}{2}(\MAT I_{p^2}+\MAT K_{pp})(\MAT\Sigma^{-1}\oplus \MAT\Sigma^{-1/2})^{-1}\right).
    \label{eq:10}
\end{equation}
By injecting Equation \eqref{eq:10} into Equation \eqref{eq:7} and using some matrix manipulations, we have:
\begin{multline}
    \vvec\left(\frac{\partial \MAT S}{\partial\MAT \Sigma^{-1}}\right)
    = -\frac{1}{2}(\MAT K_{pp}\otimes\MAT K_{pp})(\MAT I_p\otimes\MAT K_{pp}\otimes\MAT I_p)\\
    \left[ \left((\MAT I_{p^2}+\MAT K_{pp})(\MAT I_{p^2}+\MAT\Sigma^{1/2}\otimes\MAT \Sigma^{-1/2})^{-1}\MAT K_{pp}\right) \otimes \MAT K_{pp}\right.\\
    +\left. \left((\MAT I_{p^2}+\MAT K_{pp})(\MAT I_{p^2}+\MAT\Sigma^{-1/2}\otimes\MAT \Sigma^{1/2})^{-1}\right) \otimes \MAT I_{p^2}\right]
    \vvec\left(\MAT S\otimes \MAT\Sigma\right).
\end{multline}
Thus,
\begin{equation}
   \vvec(\otimes^{k-1}\MAT S)\otimes \vvec\left(\frac{\partial\MAT S}{\partial\MAT \Sigma^{-1}}\right) = -\frac{1}{2}\left(\MAT I_{p^{2k-2}} \otimes \MAT G_{\MAT\Sigma}\right)
  \left\{ \vvec\left(\otimes^{k-1}\MAT S\right)\otimes \vvec\left(\MAT S\otimes \MAT\Sigma\right)\right\}.
   \label{eq:60}
\end{equation}
where 
\begin{multline*}
    \MAT G_{\MAT\Sigma} = (\MAT K_{pp}\otimes\MAT K_{pp})(\MAT I_p\otimes\MAT K_{pp}\otimes\MAT I_p)
    \left[\left((\MAT I_{p^2}+\MAT K_{pp})(\MAT I_{p^2}+\MAT\Sigma^{1/2}\otimes\MAT \Sigma^{-1/2})^{-1}\MAT K_{pp}\right) \otimes \MAT K_{pp} \right.\\
    +\left. \left((\MAT I_{p^2}+\MAT K_{pp})(\MAT I_{p^2}+\MAT\Sigma^{-1/2}\otimes\MAT \Sigma^{1/2})^{-1}\right) \otimes \MAT I_{p^2}\right],
\end{multline*} which belongs to $\mathbb{R}^{p^4\times p^4}$.
Moreover,
\begin{equation}
\begin{split}
    &\vvec\left(\otimes^{k-1}\MAT S\right)\otimes \vvec\left(\MAT S\otimes \MAT \Sigma\right)\\
    &=(\MAT I_{p^{k-1}} \otimes\MAT  K_{p^{k-1},p^2}\otimes\MAT  I_{p^2}) \vvec((\otimes^k \MAT S) \otimes \MAT \Sigma)\\
    &=(\MAT I_{p^{k-1}} \otimes \MAT K_{p^{k-1},p^2}\otimes \MAT I_{p^2}) (\MAT I_{p^{k}} \otimes\MAT  K_{p,p^{k}}\otimes \MAT I_p)\left\{\vvec(\otimes^k \MAT S) \otimes \vvec(\MAT \Sigma)\right\}
    \label{eq:vec}
    \end{split}
\end{equation}

Combining Equation \eqref{eq:eqvec}, Equation \eqref{eq:vec} and Equation \eqref{eq:60}:
\begin{equation}
    \vvec\left(\frac{\partial E[\otimes^k \MAT S]}{\partial\MAT \Sigma^{-1}}\right) = -\frac{1}{2} \MAT J_{(k),\MAT\Sigma} \  \{\vvec(E[\otimes^k \MAT S])\otimes \vvec(\MAT\Sigma)\}
    \label{eq:eqvecexpec}
\end{equation}
where $\MAT J_{(k),\MAT\Sigma} = \left[\sum_{l=0}^{k-1}(\MAT H_{(k,l)}\otimes \MAT H_{(k,l)})\right](\MAT I_{p^{k-1}}\otimes \MAT K_{p^2,p^{k-1}}\otimes \MAT I_{p^2})(\MAT I_{p^{2k-2}}\otimes\MAT G_{\MAT\Sigma})(\MAT I_{p^{k-1}} \otimes \MAT K_{p^{k-1},p^2}\otimes \MAT I_{p^2}) (\MAT I_{p^{k}} \otimes \MAT K_{p,p^{k}}\otimes \MAT I_p)$.
By injecting Equation \eqref{eq:eqvecexpec} in the recursive equation \eqref{eq:kronW_rec},
\begin{equation*}
   \vvec(E[\otimes^{k+1}\MAT S]) =  \left[n (\MAT I_{p^k}\otimes \MAT K_{p,p^k} \otimes \MAT I_p ) +\MAT J_{(k),\MAT\Sigma} \right] \MAT K_{p^{2k},p^2}\left\{\vvec(\MAT\Sigma)\otimes \vvec(E[\otimes^k\MAT S]) \right\},
   \label{eq:kronW_rec1}
\end{equation*}
Let $\MAT M_{(k),\MAT\Sigma} =  \left[n (\MAT I_{p^k}\otimes \MAT K_{p,p^k} \otimes \MAT I_p ) +\MAT J_{(k),\MAT\Sigma} \right] \MAT K_{p^{2k},p^2}$. Hence, using the fact that $E[\MAT S]=n\MAT\Sigma$,
\begin{equation}
   \vvec(E[\otimes^{k}\MAT S]) =  \prod_{l=0}^{k-1}\left[\MAT I_{p^{2l}}\otimes\MAT M_{(k-1-l),\MAT\Sigma}\right] \left\{\otimes^k \vvec(\MAT\Sigma)\right\}.
   \label{eq:kronW_rec_}
\end{equation}

Notice that the resulting formula derives explicitly and non-recursively the $k^\textup{th}$ Kronecker moment for $\MAT S\sim \mathcal{W}(n,\MAT\Sigma)$. However, the expression of $\MAT J_{(k),\MAT \Sigma}$ depends on $\MAT\Sigma$, which is not an elegant way to compute arbitrary orders. To overcome this issue, we use the invariance of the Wishart distribution. In fact, $\MAT R = \MAT \Sigma^{-1/2}\MAT S  \MAT \Sigma^{-1/2}\sim \mathcal{W}(n,\MAT I_p)$ and it can be shown with recursion that $\otimes^k (\MAT\Sigma^{1/2}\MAT R\MAT\Sigma^{1/2}) = (\otimes^k \MAT\Sigma^{1/2})(\otimes^k \MAT R)(\otimes^k \MAT\Sigma^{1/2})$. Thus, it suffices to derive the Kronecker moments for $\MAT\Sigma=\MAT I_p$. Using Equation \eqref{eq:kronW_rec_},
\begin{equation*}
   \vvec(E[\otimes^{k}\MAT R]) = \prod_{l=0}^{k-1}\left[\MAT I_{p^{2l}}\otimes\MAT M_{(k-1-l)}\right] \left\{\otimes^k \vvec(\MAT I_p)
 \right\},  
\end{equation*}
where $\MAT M_{(k)} = \MAT M_{(k),\MAT I_p}$ and,
$\MAT G_{\MAT I_p} = \frac{1}{2}(\MAT K_{pp}\otimes \MAT K_{pp})(\MAT I_p\otimes \MAT K_{pp}\otimes \MAT I_p)\left[(\MAT I_{p^2}+\MAT K_{pp})\otimes (\MAT I_{p^2}+\MAT K_{pp})\right]$. Therefore, 
\begin{equation}
   \vvec(E[\otimes^{k}\MAT S]) = \left(\otimes^{2k}\MAT \Sigma^{1/2}\right) \prod_{l=0}^{k-1}\left[\MAT I_{p^{2l}}\otimes\MAT M_{(k-1-l)}\right] 
   \left\{\otimes^k \vvec(\MAT I_p)\right\}.
   \label{eq:kronW_rec_1
   fin}
\end{equation}
\end{proof}

\begin{proof}[\textbf{Proof of Proposition~\ref{prop:kronEW}}]
    From Theorem~\ref{thm:stoc rep}, one has $\MAT S = \mathcal{Q} \MAT\Sigma^{1/2}\MAT V\MAT \Sigma^{1/2}$ with $\MAT V \stackrel{(d)}{=} \MAT U \MAT U^\top$.
    Consequently, we get $\otimes^k\MAT S = \mathcal{Q}^k \otimes^k(\MAT \Sigma^{1/2}\MAT V \MAT\Sigma^{1/2})$, and
    \begin{equation}
        \vvec(E[\otimes^k\MAT S])
        = E[\mathcal{Q}^k] \vvec(E[\otimes^k(\MAT \Sigma^{1/2}\MAT V \MAT\Sigma^{1/2})])
        = m_k^{(h_{np})}\vvec(E[\otimes^k(\MAT \Sigma^{1/2}\MAT V \MAT\Sigma^{1/2})]).
        \label{eq:kronEWbis}
    \end{equation}
    It remains to determine $\vvec(E[\otimes^k(\MAT \Sigma^{1/2}\MAT V \MAT\Sigma^{1/2})])$.
    To do so, we leverage the fact that $\vvec(E[\otimes^k\MAT S])$ is known in the Wishart case.
    Indeed, in this specific case, using $m_k^{(h_{np})}=2^k\frac{\Gamma(\frac{np}{2}+k)}{\Gamma(\frac{np}{2})}$ and injecting~\eqref{eq:kronW} from Proposition~\ref{prop:kronW} in~\eqref{eq:kronEWbis} yields
    \begin{equation*}
        \left(\otimes^{2k}\MAT \Sigma^{1/2}\right) \prod_{l=0}^{k-1}\left[\MAT I_{p^{2l}}\otimes\MAT M_{(k-1-l)}\right] \left(\otimes^k \vvec(\MAT I_p)\right)
        =
        2^k\frac{\Gamma(\frac{np}{2}+k)}{\Gamma(\frac{np}{2})} \vvec(E[\otimes^k(\MAT \Sigma^{1/2}\MAT V \MAT\Sigma^{1/2}]).
    \end{equation*}
    This is enough to conclude.
\end{proof}

\subsection{Inverse Elliptical Wishart distribution}
\label{app:Inverse Elliptical Wishart distribution}
\begin{proof}[\textbf{Proof of Proposition~\ref{prop:esp inv}}]
    From Theorem~\ref{thm:stoc rep}, $\MAT S = \frac{1}{\mathcal{Q}}\MAT\Sigma^{1/2}(\MAT U\MAT U^\top)^{-1}\MAT\Sigma^{1/2}$.
    Since $\MAT U$ and $\mathcal{Q}$ are independent, $$E[\MAT S]= E\left[\frac{1}{\mathcal{Q}}\right]\MAT\Sigma^{1/2}E[(\MAT U\MAT U^\top)^{-1}]\MAT\Sigma^{1/2}.$$
    The result is obtained using point 4. from Proposition~\ref{prop:uniform}.
\end{proof}

\begin{proof}[\textbf{Proof of Proposition~\ref{prop:var inv}}]
    From Theorem~\ref{thm:stoc rep}, $\MAT S = \frac{1}{\mathcal{Q}}\Sigma^{1/2}(\MAT U\MAT U^\top)^{-1}\Sigma^{1/2}$.
    Therefore, $\vvec(\MAT S)= \frac{1}{\mathcal{Q}} (\MAT\Sigma^{1/2}\otimes \MAT\Sigma^{1/2}) \vvec((\MAT U\MAT U^\top)^{-1})$.
    Since $\MAT U$ and $\mathcal{Q}$ are independent,
    \begin{equation}
        E[\vvec(\MAT S)\vvec(\MAT S)^\top]= E\left[\frac{1}{\mathcal{Q}^2}\right](\MAT\Sigma^{1/2}\otimes \MAT\Sigma^{1/2}) 
        E[\vvec((\MAT U\MAT U^\top)^{-1})\vvec((\MAT U\MAT U^\top)^{-1})^\top](\MAT\Sigma^{1/2}\otimes \MAT\Sigma^{1/2}).
    \end{equation}
    Point 5. of Proposition~\ref{prop:uniform} concludes the proof.
\end{proof}

\begin{proof}[\textbf{Proof of Proposition~\ref{prop:kronIEW}}]
    From Theorem~\ref{thm:stoc rep}, $\MAT S = \frac{1}{\mathcal{Q}} \MAT\Sigma^{1/2}\MAT V^{-1}\MAT \Sigma^{1/2}$ with $\MAT V \stackrel{(d)}{=} \MAT U \MAT U^\top$.
    Then, $\otimes^{k+1}\MAT S = \frac{1}{\mathcal{Q}^{k+1}} \otimes^{k+1}(\MAT \Sigma^{1/2}\MAT V^{-1} \MAT\Sigma^{1/2})$ and,
    \begin{equation}
        \vvec(E[\otimes^{k+1}\MAT S])
        = E[\mathcal{Q}^{-k-1}] \vvec(E[\otimes^{k+1}(\MAT \Sigma^{1/2}\MAT V^{-1} \MAT\Sigma^{1/2}])
        = m_{-k-1}^{(h_{np})}\vvec(E[\otimes^{k+1}(\MAT \Sigma^{1/2}\MAT V^{-1} \MAT\Sigma^{1/2}]).
        \label{eq:kronEIWbis}
    \end{equation}    
    It remains to obtain $\vvec(E[\otimes^{k+1}(\MAT \Sigma^{1/2}\MAT V^{-1} \MAT\Sigma^{1/2})])$.
    Similarly to the proof of Proposition~\ref{prop:kronEW}, we exploit the fact that $\vvec(E[\otimes^{k+1}\MAT S])$ is known in the Inverse Wishart case.
    In this specific case, using $E[\mathcal{Q}^{-k-1}]=2^{-k-1}\frac{\Gamma\left(\frac{np}{2}-k-1\right)}{\Gamma\left(\frac{np}{2}\right)}$ and injecting~\eqref{eq:KronIW} in~\eqref{eq:kronEIWbis} yields
    \begin{equation*}
        \frac{1}{n-p-1} \prod_{i=0}^{k-1}(\MAT I_{p^{2i}}\otimes \MAT A(k-i))(\otimes^{k+1}\vvec(\MAT \Sigma))
        =
        2^{-k-1}\frac{\Gamma\left(\frac{np}{2}-k-1\right)}{\Gamma\left(\frac{np}{2}\right)}
        \vvec(E[\otimes^{k+1}(\MAT \Sigma^{1/2}\MAT V^{-1} \MAT\Sigma^{1/2})]).
    \end{equation*}
    This is enough to conclude.
\end{proof}

\section{Asymptotic behavior of Elliptical Wishart distributions}
\label{app:Asymptotic behavior of Elliptical Wishart distributions}
\begin{proof}[\textbf{Proof of Proposition~\ref{prop:asymp_normalized_wishart}}]
    Let $\MAT S \sim \mathcal{W}(n,\MAT I_p)$ and $\MAT V := \frac{\MAT S}{\tr(\MAT S)}$. By definition, $\MAT V \sim \mathcal{NW}(n,p)$. We consider the function $\psi:\mathbb{R}^{p^2}\to\mathbb{R}^{p^2}, \VEC x \mapsto \psi(\VEC x)=\frac{\VEC x}{\tr(\MAT X)}=\frac{\VEC x}{\VEC x^\top \vvec(\MAT I_p)}$ where $\VEC x = \vvec(\MAT X)$. The function $\psi$ is differentiable at each $\VEC x$ such that $\tr(\MAT X) \neq 0$. Thus, using the Delta method,
\begin{equation}
    \sqrt{n}\ \left(\psi\left(\frac{\vvec(\MAT S)}{n}\right)- \psi(\vvec(\MAT I_p))\right) \overset{\textup{(d)}}{\underset{n\to+\infty}{\longrightarrow}} \mathcal{N}_{p^2}(0,\MAT \Gamma)
\end{equation}
with $\MAT\Gamma = D\psi(\vvec(\MAT I_p))(\MAT I_{p^2}+\MAT K_{pp})D\psi(\vvec(\MAT I_p))^\top$.
We can show that $D\psi(\VEC x)= \frac{\MAT I_{p^2}}{\tr(\MAT X)} - \frac{\VEC x \vvec(\MAT I_p)^\top}{\tr(\MAT X)^2}$. 
This ends the proof.
\end{proof}

\begin{proof}[\textbf{Proof of Proposition~\ref{prop:asymp_ew}}]
Let $u_n= \sqrt{n}(\mathcal{Q}_n/\mu_n-1) $ and $\VEC v_n= \sqrt{n}\vvec\left(\MAT V_n -\frac{\MAT I_p}{p}\right)$.
    Since $\mathcal{Q}_n$ and $\MAT V_n$ are independent, $u_n$ and $\VEC v_n$ are independent, too. In addition, they are asymptotically normal; we denote $u$ and $\VEC v$ their limits (in terms of weak convergence), respectively. Thus, $(u_n,\VEC v_n) \overset{\textup{(d)}}{\underset{n\to+\infty}{\longrightarrow}} (u,\VEC v)$,
    such that 
    $u$ and $\VEC v$ are independent. On the other hand, a direct consequence of Proposition \ref{prop:asymp_normalized_wishart} is that $\MAT V_n \overset{\textup{(P)}}{\underset{n\to+\infty}{\longrightarrow}}\frac{\MAT I_p}{p}$. Hence, using the continuous mapping theorem, we get $g(u_n,\VEC v_n,\vvec(\MAT V_n)) \overset{\textup{(d)}}{\underset{n\to+\infty}{\longrightarrow}}g\left(u,\VEC v,\vvec\left( \frac{\MAT I_p}{p}\right)\right)$ with $g:(u,\VEC v,\VEC w)\mapsto \VEC v + u \VEC w $. Finally, thanks to Equation \eqref{eq:asymptotic_storep}, we obtain:
    $$\sqrt{n}\vvec\left( \frac{p\MAT S_n}{ m_1^{(h_{np})}}-  \MAT \Sigma \right) 
    \overset{\textup{(d)}}{\underset{n\to+\infty}{\longrightarrow}} u\vvec(\MAT \Sigma)  +  p(\MAT \Sigma^{\frac{1}{2}}\otimes \MAT \Sigma^{\frac{1}{2}}) \VEC v,$$
    which concludes the first point.\\
    For the second claim, note that $\tr(\MAT\Sigma^{-1}\MAT S_n)=\mathcal{Q}_n\tr(\MAT V_n)=\mathcal{Q}_n$ exactly; hence $\mathcal{Q}_n/\mu_n=\frac1p\tr(\MAT\Sigma^{-1}p\MAT S_n/\mu_n)$, and conversely $p\MAT V_n\overset{\textup{(P)}}{\underset{n\to+\infty}{\longrightarrow}}\MAT I_p$.
\end{proof}

\begin{proof}[\textbf{Proof of Corollary \ref{cor:exp}}] (b) Asymptotic behavior of $\mathcal{KW}(n,\MAT\Sigma,\alpha,\beta,R)$: From Table~\ref{table:multivariate_fim}, we have $\mathcal{Q}\sim \Gamma\left(R^{-1/\beta},\alpha-1+\frac{np}{2},\beta\right)$. By definition of the generalized Gamma distribution, this is equivalent to $2R \mathcal{Q}^\beta\sim\chi^2_{\frac{np+2(\alpha-1)}{\beta}}$. 
Using the asymptotic behavior of the Chi-squared distribution (\emph{i.e.} $\frac{\chi^2_d-d}{\sqrt{2d}}\overset{\textup{(d)}}{\underset{d\to+\infty}{\longrightarrow}}\mathcal{N}(0,1)$ \cite{george1990statistical}) with the application of the Delta method, we deduce that $\sqrt{n}(\mathcal{Q}_n/\mu_n-1) \overset{\textup{(d)}}{\underset{n\to+\infty}{\longrightarrow}}\mathcal{N}(0,\sigma^2)$ with $\sigma^2=\frac{2}{p\beta}$. Thus, the application of Proposition \ref{prop:asymp_ew} leads to \eqref{eq:asymp_kw}.\\ 
Using the Laplace asymptotic expansion of $\Gamma(x)$ when $x \to + \infty$ (see \cite{laplace}),
we can easily show that for $y>0$, $\gamma(x,y)=\frac{\Gamma(x+y)}{\Gamma(x)}\underset{x\to+\infty}{=} x^y+\frac{y(y-1)}{2}x^{y-1}+o(x^{y-1})$. Consequently,
$\gamma\left(\frac{np}{2\beta}+\frac{\alpha-1}{\beta},\frac{k}{\beta}\right)  \underset{n\to+\infty}{=}\left(\frac{np}{2\beta}\right)^{\frac{k}{\beta}}
+\frac{k}{2\beta}\left(\frac{k+2(\alpha-1)}{\beta}-1\right) \left(\frac{np}{2\beta}\right)^{\frac{k}{\beta}-1}
+o(n^ {\frac{k}{\beta}-1})$.
In particular, for $k=1$, this yields $\mu_n \underset{n\to+ \infty}{\sim}\left(\frac{np}{2R\beta}\right)^{1/\beta}$. Therefore, the Slutsky theorem yields the asymptotic distribution of $n^{\frac{3}{2}-\frac{1}{\beta}}\vvec\left( \frac{\MAT S_n}{n}- \frac{\mu_n}{np} \MAT \Sigma \right)$.\\  

(c) Asymptotic behavior of $t$-$\mathcal{W}(n,\MAT\Sigma,\nu)$:  From Table~\ref{table:multivariate_fim}, $\mu_n=np \frac{\nu}{\nu-2}$and,  $\frac{\mathcal{Q}_n}{np}\sim \mathcal{F}(np,\nu)$. Then, $\frac{np}{\mathcal{Q}_n}\sim \mathcal{F}(\nu,np)$, which yields $\frac{\nu np}{\mathcal{Q}_n} \overset{\textup{(d)}}{\underset{n\to+\infty}{\longrightarrow}}\chi^2_\nu$. Using Slutsky, $\frac{\MAT S_n}{n} \overset{\textup{(d)}}{\underset{n\to+\infty}{\longrightarrow}} \frac{\nu}{\chi^2_\nu} \MAT \Sigma$.
\end{proof}

\end{document}